\documentclass[lettersize,journal]{IEEEtran}

\usepackage[cmex10]{amsmath}
\usepackage[caption=false,font=normalsize,labelfont=sf,textfont=sf]{subfig}
\usepackage{textcomp}
\usepackage{stfloats}
\usepackage{url}
\usepackage{verbatim}
\usepackage{amsmath} 
\usepackage{amsthm}
\usepackage{amsfonts,mathrsfs}
\usepackage{amssymb}
\usepackage{color}
\usepackage[table]{xcolor}
\usepackage{graphicx}
\usepackage{epsfig}
\usepackage{algorithm,algorithmicx}
\usepackage{algpseudocode}
\usepackage[hidelinks]{hyperref}
\usepackage{marginnote}
\usepackage{empheq}
\usepackage{bm}
\usepackage{accents}
\usepackage{cite}
\usepackage{balance}
\usepackage{multirow}
\usepackage{soul}
\usepackage{enumitem}
\usepackage{marginnote}

\newtheorem{definition}{Definition}{\it}{}
{\it}{}
{\it}{}
\newtheorem{proposition}{Proposition}{\it}{}
\newtheorem{lemma}{Lemma}{\it}{}
\newtheorem{theorem}{Theorem}{\it}{}
\newtheorem{remark}{Remark}{\it}{}
\newtheorem{assumption}{Assumption}{\it}{}

\newcommand{\mc}{\mathcal}
\newcommand{\bb}{\mathbb}

\DeclareMathAlphabet{\mathbbmsl}{U}{bbm}{m}{sl}

\newcommand{\col}{\operatorname{col}}

\newcommand{\0}{\mathbf{0}}
\newcommand{\1}{\mathbf{1}}

\newcommand{\Rmnum}[1]{\expandafter\@slowromancap\romannumeral #1@}

\newcommand{\qedd}{\ensuremath{\hfill \blacksquare}}

\begin{document}
	%
	\title{Ensuring Grid-Safe Forwarding of Distributed Flexibility in Sequential DSO-TSO Markets }
	
	\author{Wicak Ananduta, Anibal Sanjab, and Luciana Marques 	
		\thanks{The authors are with the Flemish Institute for Technological Research (VITO) and
			EnergyVille, 3600 Genk, Belgium. Email addresses:
			{\{wicak.ananduta, anibal.sanjab, luciana.marques\}@vito.be}.}
		\thanks{This work is supported by the EU H2020 research and innovation programme under grant agreement No 957739 (OneNet project) and the energy transition funds project Alexander organized by the Belgian FPS economy.}}
	
	\markboth{Journal of \LaTeX\ Class Files,~Vol.~14, No.~8, August~2021}%
	{Shell \MakeLowercase{\textit{et al.}}: A Sample Article Using IEEEtran.cls for IEEE Journals}
	
	
	\maketitle

	\begin{abstract}
		This paper investigates sequential	flexibility markets consisting of a first market layer for distribution system operators (DSOs) to procure local flexibility to resolve their own needs (e.g., congestion management) followed by a second layer, in which the transmission system operator (TSO) procures remaining flexibility forwarded from the distribution system layer as well as flexibility from its own system for providing system services. As the TSO does not necessarily have full knowledge of the distribution grid constraints, this bid forwarding can cause an infeasibility problem for distribution systems, i.e.,  cleared distribution-level bids in the TSO layer might not satisfy local network constraints. To address this challenge, we formally introduce and examine three methods aiming to enable the grid-safe use of distribution-located resources in markets for system services, namely: a corrective three-layer market scheme, a bid  prequalification/filtering method, and a novel bid aggregation method. Technically, we provide conditions under which these methods can produce a grid-safe use of distributed flexibility. We also characterize the efficiency of the market outcome under these methods and reflect on their practicalities. Finally, we carry out a representative case study to evaluate the performances of the three methods, focusing on economic efficiency, grid-safety, and computational load.
	\end{abstract}
	
	\section{Introduction}
	\IEEEPARstart{T}{he} significant increase in distributed generation and storage, coupled with digitalization and energy management systems,  has opened up promising opportunities for system operators (SOs) to acquire flexibility from distribution networks through  coordinated market mechanisms \cite{gerard2018coordination,vicente2019evaluation,givisiez2020review}.  In such a flexibility market, flexibility service providers (FSPs), i.e., any type of market players with the ability to modify their power consumption or generation levels, submit their flexibility bids in order to fulfill certain system services such as balancing by the transmission SO (TSO) and congestion management by the TSO and/or distribution SO (DSO) \cite{d7.1onenet}. FSPs can be generation sources or consumers at any voltage level with controllable flexible assets, willing to adjust their generation or consumption patterns. Generation sources can, in this respect, increase or decrease their generation levels at a certain time instance to provide, respectively, upward or downward flexibility. Equivalently, consumers can increase or decrease their consumption during a certain time period to provide, respectively, downward or upward flexibility.  Note that, we focus on \emph{steady-state flexibility} \cite{riaz2022modelling}, specifically an active power adjustment with respect to a baseline.

	Establishing these flexibility market mechanisms to organize the procurement by the TSO and the DSOs (typically coined TSO-DSO coordination schemes) has been increasingly advocated in policy frameworks~\cite{Ceer} and remains a focal point of research in the scientific community~\cite{sanjab2022tso,sanjab2023joint,marques2023grid,lecadre2019game,vicente2019evaluation,papavasiliou2018coordination,JointTSODSO_Roos,FlexMarkets_Review}, as well as in numerous international demonstration initiatives~\cite{coordinet, schittekatte2020flexibility, onenet}. For instance, \cite{troncia2023market} proposes a  general framework for the design and analysis of TSO-DSO coordinated markets and studies existing demonstrators. 
	
	The common market scheme, in which all SOs can jointly procure flexibility in a co-optimized manner while taking the network constraints of all the involved grids into account, has been shown to be the most efficient scheme in terms of procurement costs \cite{marques2023grid}. However, it may not always be feasible in practice, due to the need {for} sharing full network information among the SOs and the technical and privacy challenges that this entails \cite{sanjab2023joint,marques2023grid,coordinetD62}. Additionally, although different decomposition methods have been proposed to deal with this issue in the common market scheme, e.g., in \cite{caramanis2016co-optimization,tsaousoglou2023integrating,gupta2023amalgamating}, they result in market clearing mechanisms that require iterations and multiple rounds of communication among the system operators.
	
	An alternative market scheme, which reflects well the current practice unlike the common market scheme, is the multi-layer flexibility market, which is a sequential market, where first DSOs procure flexibility resources connected to their own grids (Layer 1) and then the TSO follows (Layer 2). Therefore, unused bids from the DSO markets can be forwarded to the TSO market, thereby enhancing the offered flexibility value stacking potential.  
	Notably, this market scheme is being increasingly implemented in pilot projects across Europe (e.g. SmartNet \cite{smartnet}, CoordiNet \cite{coordinetD62}, and Interrface \cite{interrface}). Meanwhile, an early study on suitable market and regulatory conditions for bid forwarding in such a market scheme is conducted in \cite{bindu2023bid}.
	
	In this context, a crucial issue can arise when the TSO lacks awareness of the grid constraints within the distribution networks, potentially due to data privacy concerns.  Consequently, these constraints cannot be considered when the TSO clears its market despite the participation of distribution-level providers.  This raises a critical question: How can one ensure that, when the offered flexibility from the distribution systems {is cleared in TSO-layer markets, this would} not compromise the operation of the distribution networks, such as by causing local grid congestion? 
	
	As attempts to solve this issue, three approaches, namely a three-layer market scheme, a bid prequalification/filtering method, and a bid aggregation method, have been conceptually proposed  but not mathematically formulated in the aforementioned large-scale pilot projects. The three-layer market scheme, introduced in \cite{coordinetD62}, is a corrective market-based approach that involves the addition of subsequent local markets (as Layer 3) to resolve local congestions that might arise due to the outcome of the TSO market (Layer 2). This approach is akin to re-dispatch methods for congestion management \cite{bjorndal2013congestion,hermann2019congestion,pantovs2020market}. Differently, \cite{interrface} proposes a bid filtering concept, where the DSOs discard bids prior to forwarding to ensure that only safe-to-be-cleared bids are forwarded to the TSO market. The filtering process can take into account the bid prices when discarding bids for improved efficiency.  
		This scheme can {then also} directly complement the prequalification process needed to ensure the technical requirement  satisfaction of the next layer market \cite{bindu2023bid,attar2022congestion}.  Finally, \cite{papavasiliou2018coordination,papavasiliou2020hierarchical,mezghani2021coordination} discuss a bid aggregation method based on the residual supply function (RSF) approach where {each lower-layer market operator (a DSO in our case)} forwards (discrete levels of) safe aggregated bids {in the form of interface flows} along with {the linear approximation \cite{papavasiliou2018coordination} or dual-price-based approximation \cite{papavasiliou2020hierarchical,mezghani2021coordination} of} the RSF. Aggregating flexibility has been seen as a promising solution to facilitate the integration of distributed energy resources \cite{capitanescu2018tso,riaz2022modelling,khodabakhsh2023designing}. Furthermore, the RSF approach has been numerically tested in different hierarchical market schemes, as reported in \cite{nside2021} and \cite{papavasiliou2022interconnection}.
	
	The existing literature has provided valuable contributions in terms of conceptualizing these methods and providing first analyses of their suitability. However, to the best of our knowledge, no rigorous mathematical analyses of their grid-safety guarantees and their impacts on market efficiency have been provided. Therefore, in this paper, our focus is primarily on studying two aspects: 1) to what extent each method guarantees the safe operation of the distribution grids, and 2) how their performances in terms of efficiency and computational complexity compare. Our main contributions are the following:
	\begin{enumerate}[leftmargin=*]
		\item We develop  mathematical models  of the three methods, hence formally defining them. We are the first to provide mathematical models of the three-layer market scheme and bid prequalification method. Furthermore,  for the RSF-based bid aggregation method, we propose an improvement by using the exact but discretized optimal primal costs of lower layer problems as the RSFs, as shown in Algorithm \ref{alg:bid_aggregation} and Remark \ref{rem:outer_loop_alg2}, as opposed to approximated RSFs as in  \cite{papavasiliou2018coordination,papavasiliou2020hierarchical,mezghani2021coordination}. This modification allows for controlling the suboptimality of the market outcomes. We note that our formulations of these three methods are tailored to sequential flexibility markets.
		\item Then, we analyze the properties of these three methods, i.e., we identify conditions for achieving grid safety and evaluate their efficiencies, summarized as follows:
		\begin{itemize}[leftmargin=*]
			\item We show that the three-layer market can guarantee grid-safe (but possibly suboptimal) solutions on the condition that the DSO-layer markets are sufficiently liquid.
			\item We show that the bid prequalification method remarkably only requires {the properties of radial distribution systems and common pricing principles of upward and downward flexibility offers to guarantee grid-safety} (Proposition \ref{prop:grid_safe_alg1}.\ref{prop:grid_safe_alg1_point}). We also  characterize its sub-optimality upper and lower bounds (Propositions \ref{prop:grid_safe_alg1}.\ref{prop:subopt_alg1_1}--\ref{prop:subopt_alg1_2}).
			\item  
			We show that the bid aggregation method is guaranteed to produce grid-safe cleared bids (Proposition \ref{prop:grid_safe_alg2}). We also provide a theoretical upper bound on the suboptimality of {this} method (Theorem \ref{th:suboptimality_rsf}), which is shown to depend on the step sizes of the RSF functions.  
			The advantages of the bid aggregation method come at the cost of having to solve a mixed-integer linear program at the TSO layer, which  can hinder the bid aggregation method's practical implementation potential, as compared to the linear programming approaches offered by its two counterparts. 
		\end{itemize}  
	\end{enumerate}
	
	Our mathematical derivations and analyses are completed by an extensive numerical study for comparing the three proposed methods, by using an interconnected system comprised of IEEE and Matpower test cases. The numerical results corroborate our analytical findings and yield key insights on the potential of these methods. 
	Finally, we include a reflection on the practical implementability and coherence with existing EU regulation (as we primarily focus on European markets), highlighting practical challenges that can be faced by the bid aggregation method, which must be weighted against its theoretical superiority.
	
	The paper is organized as follows. Section \ref{sec:seq_market} introduces the sequential  market model and its feasibility problem. Sections \ref{sec:three-layer}--\ref{sec:bid_aggregation}, respectively, present the models of the three-layer, bid prequalification, and bid aggregation methods, along with our analyses regarding grid-safety, market efficiency, computational complexity, as well as practical challenges for their implementation.  Section \ref{sec:num_sim} presents the numerical case study, while Section \ref{sec:conclusion} summarizes our comparison and provides concluding remarks. The proofs of all technical statements are provided in the Appendix.
	
	\paragraph*{Notation} We denote the set of real numbers by $\bb R$ and the set of positive real numbers by $\bb R_+$. Vectors are denoted by bold symbols. The operator $\col(\cdot)$ concatenates its argument column-wise. The cardinality of a set is denoted by $|\cdot|$. An $n$-dimensional hyperbox is denoted by $[\bm h^{\min}, \bm h^{\max}] \subset \bb R^n$, where $\bm h^{\min},\bm h^{\max} \in \bb R^n$ denote the lower and upper bound vectors.
	
	\section{Sequential Flexibility Markets}
	\label{sec:seq_market}
	
	\subsection{Market Model}
	\label{sec:market_model}
	In a sequential (multi-layer) market, we suppose that the DSOs have priority to first purchase flexibility from their local  FSPs to meet their own grid needs, e.g., for (active power) congestion management, which is the focus in this work, while abiding by their own network constraints. Subsequently, the remaining flexibility is forwarded to the transmission-layer market (run to meet the TSO needs, such as balancing), in which transmission-level FSPs also offer their flexibility. As such, both the value stacking potential of distributed-energy resources and the efficiency of flexibility markets are enhanced by allowing FSPs located in distribution systems to participate in transmission-layer markets as well \cite{marques2023grid,sanjab2023joint}. 
	
	First, we provide a mathematical formulation of the optimization problems solved to clear these sequential markets. 
	We denote by $\mc N^{\mathrm D}:=\{1,2,\dots,N^D\}$ the set of distribution systems connected to transmission busses. We use the subscript $0$ to associate variables of the TSO. Let us, then, define the local decision variables of the DSOs and TSO by $\bm x_m  {=} \col(\bm u_m,\bm d_m,\bm p_m)$, for each $m \in\mc N:= \mc N^{\mathrm{D}} \cup \{0\}$,  where $\bm u_m:=\col((u_{m,n})_{n \in \mc U_m})$,  $\bm d_m:=\col((d_{m,n})_{n \in \mc D_m})$, and $\bm p_m \in \bb R^{|\mc N_m|}$ denote the  {decision} vectors of upward flexibility {volume,} downward flexibility {volume,} and {net-injected} power, respectively, with $\mc U_m$, $\mc D_m$, and $\mc N_m$ being, {respectively,} the sets of upward FSPs, downward FSPs, and busses of network $m$.    
	Next, we denote by $z_m \in \bb R$, $\forall m \in \mc N^{\mathrm D}$, the interface power flow between the DSO-$m$ and the corresponding transmission bus,  {where a positive value denotes that the power flows from the transmission bus to the distribution network,} 
	and define  $\bm z = \col((z_m)_{m \in \mc N^{\mathrm D}})$. For simplicity, we consider that each DSO only has one interconnection point with the transmission network.  
	We can then model the sequential market clearing scheme as follows:\\
		\noindent \textbf{Layer 1:} Each DSO $m \in \mc N^{\mathrm D}$ solves the linear program (LP):
		\begin{subequations}
			\begin{align}
				& \min_{\bm x_m, z_m} \ \   \bm c_m^\top \bm x_m + c^z_m z_m \label{eq:cost_dso}\\
				&\operatorname{s.t.} \notag\\ 
				&	 {
					\sum_{n \in \mc U_{m,1}}\hspace{-2pt} u_{m,n} \hspace{-2pt}-\hspace{-4pt} \sum_{n \in \mc D_{m,1}}\hspace{-2pt} d_{m,n} \hspace{-2pt}-\hspace{-2pt} p_{m,1} \hspace{-2pt}-\hspace{-2pt} z_m \hspace{-2pt}=\hspace{-2pt} e_{m,1}, \label{eq:power_balance_dso_detail1}} \\
				&	 {
					\sum_{n \in \mc U_{m,k}}\hspace{-2pt} u_{m,n}\hspace{-2pt} -\hspace{-4pt} \sum_{n \in \mc D_{m,k}}\hspace{-2pt} d_{m,n}\hspace{-2pt} -\hspace{-2pt} p_{m,n} \hspace{-2pt}=\hspace{-2pt} e_{m, {k}}, \forall k \hspace{-2pt}\in\hspace{-2pt} \mc N_m \hspace{-2pt}\setminus\hspace{-2pt} \{1\}, \label{eq:power_balance_dso_detail2}}\\
				& C_m \bm p_m \in [\bm f_m^{\min}, \bm f_m^{\max}], \label{eq:limit_flow_dso}\\
				&  \bm u_m \in [0,\bm u_m^{\max}], \  \bm d_m \in [0,\bm d_m^{\max}], \label{eq:limit_u_d_dso} \\
				& z_m \in [z_m^{\min}, z_m^{\max}]. \label{eq:limit_z_dso}
			\end{align}
			\label{eq:layer_dso}%
		\end{subequations}%
		\indent	The objective of Problem \eqref{eq:layer_dso} is to minimize the procurement cost in \eqref{eq:cost_dso}, where the vector $\bm c_m = \col(\bm c_m^u, -\bm c_m^d, \0 )$ collects the upward ($\bm c_m^u := \col((c_{m,n}^u)_{n \in \mc U_m})\in \bb R_+^{|\mc U_m|}$) and downward ($\bm c_m^d:= \col((c_{m,n}^d)_{n \in \mc D_m})\in \bb R_+^{|\mc D_m|}$) bid prices, {while} $c_m^z \in \bb R$ denotes the interface flow price. The equality constraints  \eqref{eq:power_balance_dso_detail1}--\eqref{eq:power_balance_dso_detail2} are the power balance equations of all busses in the distribution network $m$, where, for each node $k \in \mc N_m$, $\mc U_{m,k}$, $\mc D_{m,k}$, {and $e_{m,k}$}  denote, respectively, the subsets of upward bids,  downward bids, and the net anticipated base injections, which reflects the amount of required flexibility at every node. Without loss of generality, the root node is labeled as node 1.  
		The constraints in \eqref{eq:limit_flow_dso}--\eqref{eq:limit_z_dso} respectively represent the bounds of the line power flows,  
		flexibility bids,  
		and interface flow.  
		The matrix $C_m$ in \eqref{eq:limit_flow_dso} is a nodal injection to line flow sensitivity matrix, such as the power transfer distribution factor (PTDF). We focus on the setting in which Problem \eqref{eq:layer_dso} has a non-empty feasible set, i.e., the volume of offered flexibility is collectively adequate to meeting the DSO's flexibility need (supporting the use of a flexibility market). We denote a solution to Problem \eqref{eq:layer_dso} by $(\bm x_m^{\star}, z_m^{\star})$. 
		In addition, in the next sections, we use the following compact formulation of \eqref{eq:power_balance_dso_detail1}--\eqref{eq:power_balance_dso_detail2}:
			\begin{equation}
				A_m \bm x_m + B_m z_m = \bm e_m, \label{eq:power_balance_dso}
			\end{equation}
			where $\bm e_m:=\col((e_{m,k})_{k\in \mc N_m})$, while $A_m \in \bb R^{|\mc N_m| \times (|\mc U_m|+|\mc D_m|+|\mc N_m|)}$ is a linearly-independent selection matrix and {$B_m \in \bb R^{|\mc N_m|}$} is a standard Euclidean basis vector such that \eqref{eq:power_balance_dso_detail1}--\eqref{eq:power_balance_dso_detail2} hold.
		\medskip
		
		\noindent \textbf{Layer 2:} The TSO solves the following LP:
		\begin{subequations}
			\begin{align}
				& \min \ \
				\bm c_0^\top \bm x_0 +\hspace{-4pt} \sum_{m \in \mc N^{\mathrm{D}}} \hspace{-4pt}\left(\bm c_m^{u\top} \bm u_m - \bm c_m^{d\top} \bm d_m - c^z_m z_m\right) \\
				&\operatorname{s.t.}  \notag\\  
				&	 {
					\sum_{n \in \mc U_{0,k}}\hspace{-2pt} u_{0,n} \hspace{-2pt}-\hspace{-4pt} \sum_{n \in \mc D_{0,k}}\hspace{-2pt} d_{0,n} \hspace{-2pt}-\hspace{-2pt} p_{0,k} \hspace{-2pt}+\hspace{-2pt} z_{\phi(k)} \hspace{-2pt}=\hspace{-2pt} e_{0,k}, \quad \forall k \in \mc N_0^{\mathrm c} \label{eq:power_balance_tso_detail1}} \\
				&	 {
					\sum_{n \in \mc U_{0,k}}\hspace{-2pt} u_{0,n}\hspace{-2pt} -\hspace{-4pt} \sum_{n \in \mc D_{0,k}}\hspace{-2pt} d_{0,n}\hspace{-2pt} -\hspace{-2pt} p_{0,n} \hspace{-2pt}=\hspace{-2pt} e_{0, {k}}, \quad  \forall k \hspace{-2pt}\in\hspace{-2pt} \mc N_0 \hspace{-2pt}\setminus\hspace{-2pt} \mc N_0^{\mathrm c}, \label{eq:power_balance_tso_detail2}}\\
				& \1\hspace{-1.5pt}^\top \hspace{-1.5pt} (\bm u_m \hspace{-1.5pt}+\hspace{-1.5pt} \bm u_m^\star) \hspace{-1.5pt}-\hspace{-1.5pt} \1\hspace{-1.5pt}^\top\hspace{-1.5pt}(\bm d_m  \hspace{-1.5pt}+\hspace{-1.5pt}\bm d_m^\star) + z_m \hspace{-1.5pt}=\hspace{-1.5pt} \1\hspace{-1.5pt}^\top \bm e_m ,    \forall m \hspace{-1.5pt}\in\hspace{-1.5pt} \mc N^{\mathrm D}\hspace{-1.5pt}, \label{eq:aggregated_balance_dn_tso}\\
				& C_0 \bm p_0 \in [\bm f_0^{\min}, \bm f_0^{\max}], \label{eq:limit_flow_tso} \\
				&  \bm u_0 \hspace{-2pt}\in [0,\bm u_0^{\max}],   \bm d_0 \in [0,\bm d_0^{\max}], \label{eq:limit_u_d_tso} \\
				& \bm u_m \hspace{-2pt}\in [0,{\bm u_m^{\max}\hspace{-2pt}-\hspace{-1pt}\bm u_m^{\star}}], 
				\bm d_m \hspace{-2pt}\in [0,{\bm d_m^{\max}\hspace{-2pt}-\hspace{-2pt}\bm d_m^{\star}}], \forall m \in \mc N^{\mathrm D}, \label{eq:limit_u_d_dn_tso}\\
				& z_m \in [z_m^{\min}, z_m^{\max}], \ \forall m \in \mc N^{\mathrm D}. \label{eq:limit_z_tso}
			\end{align}%
			\label{eq:layer_tso}%
		\end{subequations}%
		\indent Problem \eqref{eq:layer_tso} aims at minimizing the TSO's procurement costs considering transmission-level and distribution-level bids as well as the pricing of the interface flow. The parameters and components of the problem are defined similarly to Problem \eqref{eq:layer_dso}. {Indeed,} $\bm c_0 = \col(\bm c_0^u, -\bm c_0^d, \0 )$ collects the upward ($\bm c_0^u\in \bb R_+^{|\mc U_0|}$) and downward ($\bm c_0^d\in \bb R_+^{|\mc D_0|}$) bid prices of transmission-level FSPs;  \eqref{eq:power_balance_tso_detail1}--\eqref{eq:power_balance_tso_detail2} represent the nodal power balance equations in the transmission network, with $\mc N_0^{\mathrm c}$ being the set of transmission nodes connected to a distribution system and $\phi$ is a one-to-one mapping from $\mc N_0^{\mathrm c}$ to $\mc N^{\mathrm D}$;
		\eqref{eq:limit_flow_tso} represents the line capacity limits, and \eqref{eq:limit_u_d_tso} represents the transmission-level bid capacities. It is important to note that although distribution-level FSPs participate in this layer as well, the bid capacities are based on what remains from the first layer (as captured in \eqref{eq:limit_u_d_dn_tso}).  Moreover, the TSO only considers the aggregated power balance constraint for each distribution system as in \eqref{eq:aggregated_balance_dn_tso} instead of full network constraints of the distribution systems \eqref{eq:power_balance_dso_detail1}--\eqref{eq:limit_flow_dso}, which then includes the risk that a market clearing result of Layer 2 may cause network constraint violations for the DSOs' grids from which this flexibility originates. Let us denote a solution to the second layer by $(\bm x_0^{\star\star}$,  $(\bm u_m^{\star\star}, \bm d_m^{\star\star}, z_m^{\star\star})_{ \forall m \in \mc N^{\mathrm D}})$, which we assume exists, for any $\bm u_m^{\star} \in [0,\bm u_m^{\max}]$ and $\bm d_m^{\star} \in [0,\bm d_m^{\max}]$, for all $m \in \mc N^{\mathrm D}$. This assumption implies that the TSO can meet its flexibility needs by procuring from transmission-level FSPs only. However, a more efficient solution can potentially be obtained by allowing the participation of distribution-level FSPs. As such, based on the introduced Layer 1 and Layer 2, the \emph{practical} sequential market clearing problem can then be formally defined in Definition 1. Here, we refer to the sequential market in Definition 1 as practical in the sense that it does not require knowledge of the distribution grid models in the transmission-level market (Layer 2).
		
		\begin{definition}
			\label{def:sequential} 
				The practical sequential market clearing problem is defined by Problem \eqref{eq:layer_dso}, for each $m \in \mc N^{\mathrm D}$, {as Layer 1} and Problem \eqref{eq:layer_tso} as {Layer 2.} 
		\end{definition}

		\subsection{ {Relationship with other DSO-TSO market models}}
		
		{In addition to the practical sequential market clearing model in Definition \ref{def:sequential}, other TSO-DSO coordinated market models are introduced next due to their relevance for our analysis.}
		\begin{definition}
			\label{def:ideal_sequential} { 
				The {\emph{idealized}} sequential market clearing problem is defined by Problem \eqref{eq:layer_dso}, for each $m \in \mc N^{\mathrm D}$, {as Layer 1} and Problem \eqref{eq:layer_tso}, where \eqref{eq:aggregated_balance_dn_tso} is substituted with \eqref{eq:power_balance_dso} and \eqref{eq:limit_flow_dso}, {as Layer 2.} A solution to this problem is denoted by $(\bm x_0^{\mathrm i}$,  $(\bm x_m^{\mathrm i}, z_m^{\mathrm i})_{m \in \mc N^{\mathrm D}})$.}
		\end{definition}
		
		{In an idealized sequential market model (Definition \ref{def:ideal_sequential}), the distribution system constraints are included in the transmission-layer market (layer 2). The sequential market model in Definition \ref{def:sequential} is, then, a modified version of the idealized model, capturing the practical setting in which the TSO does not have access to the distribution systems network models, and hence, cannot include the distribution-level power flow representation and constraints in its market clearing.}
		As shown in \cite[Cor. III.2]{marques2023grid}, the optimality of {the} {idealized} sequential market model  is lower bounded by the optimal value of the common market problem, {capturing a joint TSO-DSO market clearing setting rather than a sequential multi-level structure,} defined as follows: 
		\begin{definition}
			\label{def:common_market} {
				The common market clearing problem is formulated as:
				\begin{subequations}
					\begin{align}
						\min_{\bm x_0, (\bm x_m, z_m)_{m \in \mc N^{\mathrm D}}} \ \ &  J^{\mathrm{tot}} := \bm c_0^\top \bm x_0 + \textstyle\sum_{m\in \mc N^D} \bm c_m^\top \bm x_m \label{eq:common_market_cost}\\
						\operatorname{s.t.} \ \ & \text{{\eqref{eq:power_balance_tso_detail1}, \eqref{eq:power_balance_tso_detail2},} \eqref{eq:limit_flow_tso}, and \eqref{eq:limit_u_d_tso}, } \notag\\
						& \text{{\eqref{eq:power_balance_dso_detail1}, \eqref{eq:power_balance_dso_detail2},} \eqref{eq:limit_flow_dso}, \eqref{eq:limit_u_d_dso}, and \eqref{eq:limit_z_dso}, } \forall m \in \mc N^{\mathrm D}. \notag
					\end{align}
					\label{eq:common_market}%
				\end{subequations}%
				{We denote a solution to Problem \eqref{eq:common_market} by $(\bm x_0^\circ,(\bm x_m^\circ, z_m^\circ)_{m \in \mc N^{\mathrm D}})$.} 
			}
		\end{definition}
	
			Furthermore, as we show next in Proposition \ref{prop:sequential_frag}, the optimality of {the} idealized sequential market model  can be upper bounded by the  optimal value of {another sequential DSO-TSO coordinated market model known as} the fragmented market model (Definition \ref{def:fragmented}), 	where each system operator only has access to its local FSPs but the effect of the clearing of Layer 1, in terms of interface flows, is taken into account in Layer 2.
			\begin{definition}
				\label{def:fragmented} The fragmented market clearing problem is defined by Problem \eqref{eq:layer_dso}, for each $m \in \mc N^{\mathrm D}$, as Layer 1 and Problem \eqref{eq:layer_tso}, where \eqref{eq:limit_u_d_dn_tso} is substituted with $\bm u_m = 0$, $\bm d_m=0$, for all $m \in \mc N^{\mathrm D}$ and \eqref{eq:limit_z_tso} is substituted with $z_m = z_m^\star$, {as Layer 2.} A solution to this problem is denoted by $(\bm x_0^{\mathrm f}$,  $(\bm x_m^{\mathrm f}, z_m^{\mathrm f})_{m \in \mc N^{\mathrm D}})$.
			\end{definition} 
			\begin{proposition}
				\label{prop:sequential_frag}
				Let $(\bm x_0^{\mathrm i}$,  $(\bm x_m^{\mathrm i}, z_m^{\mathrm i})_{m \in \mc N^{\mathrm D}})$ be a solution to the {idealized} sequential market model (Definition \ref{def:ideal_sequential}) and $(\bm x_0^{\mathrm f}$,  $(\bm x_m^{\mathrm f}, z_m^{\mathrm f})_{m \in \mc N^{\mathrm D}})$ be a solution to the fragmented market model (Definition \ref{def:fragmented}). Then it holds that
				\begin{equation}
					\bm c_0^\top \bm x_0^{\mathrm i} + \textstyle\sum_{m\in \mc N^D} \bm c_m^\top \bm x_m^{\mathrm i} \leq \bm c_0^\top \bm x_0^{\mathrm f} + \textstyle\sum_{m\in \mc N^D} \bm c_m^\top \bm x_m^{\mathrm f}.
					\label{eq:ineq_sequential_frag}
				\end{equation}
			\end{proposition}%
			However, these lower and upper optimality bounds do not apply to the (practical) sequential market model due to the exclusion of the distribution networks' constraints in Layer 2. In fact, there is not even a feasibility guarantee for the outcome of the practical sequential market as we discuss next.
		
		\subsection{{Grid Safety Considerations}}
		\label{sec:grid_safety_considerations}
		The sequential market (Definition \ref{def:sequential}) {considers} that the TSO does not know the grid constraints of {the} distribution systems. {In other words,} \eqref{eq:limit_flow_dso}, for all $m \in \mc N^{\mathrm D}$, are not included as constraints in Problem \eqref{eq:layer_tso}. This assumption is relevant in practice as DSOs might {be restricted from sharing} private {network} information externally (e.g., with the TSO or a third party market operator) or might simply be unable to replicate network models in external servers as investigated in \cite{coordinetD62}.  {Hence,} a solution to {Layer 2} might not be \emph{grid safe} for the distribution system, i.e., it might not satisfy the constraint in \eqref{eq:limit_flow_dso}, for some $m \in \mc N^{\mathrm{D}}$. {In contrast, grid-safe bids are formally defined next.}
		\begin{definition}[Grid-safe bids]\label{def:grid_safe}
			Bids $(\bm u_m, \bm d_m)$, for all $m \in \mc N$, are grid safe if there exist $\bm p_m$, for all $m \in \mc N$, and $\bm z$ such that  {\eqref{eq:power_balance_dso_detail1}, \eqref{eq:power_balance_dso_detail2},} \eqref{eq:limit_flow_dso}, and \eqref{eq:limit_z_dso}, for all $m \in \mc N^{\mathrm D}$, as well as  {\eqref{eq:power_balance_tso_detail1}, \eqref{eq:power_balance_tso_detail2},} and \eqref{eq:limit_flow_tso} hold.
		\end{definition}
		
		Definition \ref{def:grid_safe} technically states that a set of distribution-level bids $(\bm u_m, \bm d_m)$ are grid safe if its activation results in physical variable value $\bm p_m$ that satisfies all grid constraints. 	{We note that, in this work, we focus on congestion management constraints of distribution grids, i.e.,  \eqref{eq:power_balance_dso_detail1}--\eqref{eq:limit_z_dso}, and use a linear network model.  Therefore, the grid-unsafe events that are the focus of this work are distribution line congestions. To include other types of constraints in the definition of grid-safe bids, e.g., voltage constraints, which implies that we take into account voltage violations as another grid-unsafe events, we can either  expand the sensitivity-matrix-based model of distribution systems to also consider, e.g., net off-take to nodal voltage sensitivity matrices in a similar way to constraint \eqref{eq:limit_flow_dso} or consider different power flow models, such as a (linearized) Branch Flow model \cite{LinDistFlow}. Although more complex network models can be adapted to the bid forwarding methodologies discussed in this work, their technical analyses will require dedicated efforts, which can constitute the focus point of future work.} 
		
		{Next, we show a particular example when the cleared bids of the sequential market are grid safe. We suppose that} {the interface flows are optimally priced, i.e., in Problems \eqref{eq:layer_dso} and \eqref{eq:layer_tso},  {$c_m^z$ are chosen as an optimal dual variables of Problem \eqref{eq:common_market} associated with \eqref{eq:power_balance_tso_detail1}.} {In this case,} {forwarding any remaining distribution-level bids to Layer 2 is not needed,} as  {implied by the next proposition.} 
			\begin{proposition}
				\label{prop:optimal_price}
				Suppose that $\col((c_m^z)_{m \in \mc N^D})$ is an optimal dual variable of Problem \eqref{eq:common_market} associated with  {\eqref{eq:power_balance_tso_detail1}} (an optimal price). Then, the cleared bids $(\bm u_m^{\star}+\bm u_m^{\star\star}, \bm d_m^{\star}+\bm d_m^{\star\star}), \forall m \in \mc N^{\mathrm D},$ and $(\bm u_0^{\star\star},\bm d_0^{\star\star})$ computed by solving {Problems} \eqref{eq:layer_dso} {and} \eqref{eq:layer_tso} are an optimal solution to Problem \eqref{eq:common_market} only if $\bm u_m^{\star\star}=0$ and $\bm d_m^{\star\star}=0$, for all $m \in \mc N^{\mathrm D}$.
			\end{proposition}
			
			{Proposition \ref{prop:optimal_price} shows that under an optimal interface flow price, a necessary condition for an optimal clearing is that no distribution-level bids are cleared in {Layer 2, as the clearing of Layer 1 also returns an optimal interface flow for Layer 2, thus resulting in a common market solution.} Consequently,} the cleared bids  {in the transmission and distribution systems} are grid safe. {However,} obtaining \emph{an optimal price} 
			of the interface flows essentially requires {virtually solving Problem \eqref{eq:common_market}, which is impractical and still requires full knowledge of the distribution grid models in Layer 2.} Furthermore, {in contrary to the idealized sequential market model,} even if an optimal interface flow price is known a priori, non-zero distribution-level bids might still be cleared {in Layer 2} under the {practical sequential} market scheme (Definition \ref{def:sequential}) {due to the exclusion of the distribution-network constraints in the second layer market, as shown in an illustrative example in Section \ref{sec:num_sim}. This fact adds motivation to study approaches to deal with the infeasibility issue {arising} from {this market model.}} 
			
			\section{Corrective Method: A Three-layer Scheme}
			\label{sec:three-layer}
			A straightforward corrective method to the feasibility issue of the sequential market {can be achieved} by introducing {an additional market} layer {(rendering the sequential market scheme three-layered)} that aims at resolving any potential congestion in the distribution systems {that was caused by the solution of Layer 2,} {as conceptualized in} \cite{coordinetD62}. 
			
			\subsection{{Model Formulation}}
			{This method can then be formulated by the addition of the third layer market correction problem.} Specifically, after the TSO forwards the outcome of {Layer 2} to the DSOs, each DSO $m \in \mc N^{\mathrm D}$ solves: 
			\begin{subequations}
				\begin{align}
					\min_{\bm x_m, z_m} \ \ & \bm c_m^\top \bm x_m \notag \\
					\operatorname{s.t.} \ \ 
					&  A_m \bm \chi_m + B_m z_m = e_m, \label{eq:power_balance_dso_2}  \\
					& \text{\eqref{eq:limit_flow_dso} and }
					z_m  = z_m^{\star\star}, \label{eq:fixed_z} \\
					& \bm \chi_m = \col(\bm u_m+\bm u_m^{\star}+\bm u_m^{\star\star}, \bm d_m+\bm d_m^{\star}+\bm d_m^{\star\star}, \bm p_m), \notag\\
					&  \bm u_m \in [0,\bm u_m^{\max}-\bm u_m^{\star}-\bm u_m^{\star\star}], \notag\\
					&  \bm d_m \in [0,\bm d_m^{\max}-\bm d_m^{\star}-\bm d_m^{\star\star}],   \notag
				\end{align}
				\label{eq:layer3_dso}%
			\end{subequations}%
			where the interface flow is fixed based on the outcome of Layer 2 {as in \eqref{eq:fixed_z}.}  {Note that \eqref{eq:power_balance_dso_2} is obtained from \eqref{eq:power_balance_dso} where $\bm x_m$ is substituted with $\bm \chi_m$.} Let us denote a solution to Problem \eqref{eq:layer3_dso}  by $(\bm x_m^{\mathrm{c}},z_m^{\mathrm{c}})$, which is the correction needed to satisfy the grid constraints. We note that there is no grid safety issue for the transmission system due to the introduction of the third layer corrective market since the interface flows in the third layer are fixed to the interface flow solutions of Layer 2, as described in \eqref{eq:fixed_z}.

			\subsection{{Method Analysis}}
			{By the construction of the third-layer corrective market problem in \eqref{eq:layer3_dso}, a grid-safe cleared bid can be obtained if Problem \eqref{eq:layer3_dso} has a non-empty feasible set. In this case, the outcome of the approach, i.e., $(\bm u_m^{\star}+\bm u_m^{\star\star}+\bm u_m^{\mathrm{c}}, \bm d_m^{\star}+\bm d_m^{\star\star}+\bm d_m^{\mathrm{c}}), \forall m \in \mc N^{\mathrm D},$ and $(\bm u_0^{\star\star},\bm d_0^{\star\star})$ are grid safe as they satisfy all the network constraints. Furthermore, they must be a feasible solution to the idealized sequential market problem (Definition \ref{def:ideal_sequential}), implying that the outcome of the three-layer market scheme can only be as efficient as a set of optimal bids of the idealized sequential market. In other words, the efficiency of this approach is lower bounded by that of the idealized sequential market.} {The suboptimality (extra cost incurred) is due to constraint violations caused by the outcomes of the first two layers that must be rectified. Indeed, in the best scenario, the optimal solution of Layer 3 is $\bm x_m^{\mathrm c}=0$, meaning that no corrective action is needed, resulting in a solution as optimal as the idealized sequential market one.} Furthermore, {the grid-safe condition} implies that, {in practice,} the DSO-layer markets must be liquid enough to resolve any potential congestion caused by the clearing of Layer 2. Nevertheless, Problem \eqref{eq:layer3_dso} can be infeasible, implying that the outcome of the sequential market violates grid constraints in \eqref{eq:limit_flow_dso}, for some distribution systems, and these constraint violations cannot be corrected by the remaining DSO-located bids from Layer 2. In that case, the DSO needs to resort to alternative congestion management solutions, which could be more disruptive, such as load curtailment. 
			
			In terms of computational complexity, this method requires solving $|\mc{N}^D|$ more problems, as compared to the sequential market problem, as shown in Remark 1. As these problems are LP problems, this additional computational burden can be limited due to the efficiency of solving LP problems.
			\begin{remark}
				The extra layer requires each DSO to solve an LP, implying that the total number of LPs that must be solved is $2|\mc N^{\mathrm D}|+1$ (twice for each DSO and once for the TSO).
			\end{remark}
			
			{From a practical perspective, this method follows full market-based mechanisms, in line with, e.g., general EU policy recommendations \cite{eu2019-943} without requiring interference through bid selection nor filtering, which enhances its application potential in practice, especially in settings of high levels of distributed-flexibility participation.} However, market timings can be challenging for implementing the corrective market layer. For instance, if Layer 2 is performed 15 minutes before real-time, as in some balancing markets in Europe \cite{elia2020balancing_rules}, the DSO is left with little time to gather updated data and procure the corrective flexibility. In addition to the potential requirements of last-minute corrective actions to be taken by the DSO after the clearing of Layer 2 and before real-time operation (at instances leaving a window of less than 15 minutes for corrective actions in the third layer), the risks of having scarce flexible capacity in the distribution system, limiting the potential of corrective actions to resolve all grid violations that could be induced from Layer 2,  induce challenges to the practical application of this method. Such challenges render the method more suitable for settings with adequate distribution flexibility market liquidity and flexibility products timelines that allow appropriate time for corrective actions in Layer 3 to be procured and implemented.

			\section{{Bid Prequalification} Method}
			\label{sec:bid_filtering}

			To limit the need for corrective actions, the DSOs can filter/prequalify the bids that will be forwarded to Layer 2. This approach has been initially conceptualized in \cite{interrface} and {we propose an extension applicable for the considered market setup.} By discarding any bids that might result in congestion in the distribution systems, {this method aims to} ensure that the bids cleared in Layer 2 satisfy the DSOs' grid constraints. 
			
			\subsection{{Model Formulation}}
			\begin{algorithm}[!t]
				\caption{Bid filtering method.}
				\label{alg:bid_filtering}
				Filter upward bids in $\mc U_m=: \mc U_m^{(1)}$. 
				For $\ell = 1,2,\dots, |\mc U_m|$:
				\begin{enumerate}
					\item Solve the feasibility problem:
					\begin{equation}
						\begin{aligned}
							\operatorname{find}  \ \ & {(\bm x_m, z_m)}  \\
							\operatorname{s.t.}  \ \ & \text{ \eqref{eq:limit_flow_dso}, \eqref{eq:limit_z_dso}, and \eqref{eq:power_balance_dso_2},}\\
							& \bm \chi_m = \col(\bm u_m+\bm u_m^{\star}, \bm d_m+\bm d_m^{\star}, \bm p_m),\\
							&  u_{m,n} \in [0, u_{m,n}^{\max}- u_{m,n}^{\star}], \ \forall n \in \mc U_m^{(\ell)}, \\
							&  \bm d_m = \0, \ u_{m,n} = 0, \ \forall n \in \mc U_m^{(1)}\setminus \mc U_m^{(\ell)}. 
						\end{aligned}
						\label{eq:upward_filter}%
					\end{equation}%
					\item If Problem \eqref{eq:upward_filter} is infeasible, discard the most expensive bid,  $v_m^{(\ell)} \in \operatorname{argmax}_{v \in \mc U_m^{(\ell)}} c_{m,v}^u$, i.e., $\mc U_m^{(\ell+1)}= \mc U_m^{(\ell)}\setminus \{v_m^{(\ell)}\}$. Otherwise, stop the iterations and denote the filtered set of upward bids by $\mc U_m^{\mathrm f} := \mc U_m^{(\ell)}$.
				\end{enumerate}
				Filter downward bids in $\mc D_m=: \mc D_m^{(1)}$. 
				For $\ell = 1,2,\dots , |\mc D_m|$:
				\begin{enumerate}
					\item Solve the feasibility problem:
					\begin{equation}
						\begin{aligned}
							\operatorname{find}  \ \ & {(\bm x_m, z_m)}  \\
							\operatorname{s.t.} \ \ & \text{ \eqref{eq:limit_flow_dso}, \eqref{eq:limit_z_dso}, and \eqref{eq:power_balance_dso_2},}\\
							& \bm \chi_m = \col(\bm u_m+\bm u_m^{\star}, \bm d_m+\bm d_m^{\star}, \bm p_m),\\
							&  d_{m,n} \in [0, d_{m,n}^{\max}- d_{m,n}^{\star}], \forall n \in \mc D_m^{(\ell)}, \\
							&  \bm u_m = \0, \ d_{m,n} = 0, \forall n \in \mc D_m^{(1)}\setminus \mc D_m^{(\ell)}.
						\end{aligned}
						\label{eq:downward_filter}%
					\end{equation}%
					\item If Problem \eqref{eq:downward_filter} is infeasible, discard the least expensive bid, $v_m^{(\ell)}  \in \operatorname{argmin}_{v \in \mc D_m^{(\ell)}} c_{m,v}^d$, i.e., $\mc D_m^{(\ell+1)}= \mc D_m^{(\ell)}\setminus \{v_m^{(\ell)}\}$. Otherwise, stop the iterations and denote the filtered set of upward bids by $\mc D_m^{\mathrm f} := \mc D_m^{(\ell)}$.
				\end{enumerate}
			\end{algorithm}
			
			{We formally define this approach in Algorithm \ref{alg:bid_filtering},} performed by each DSO after obtaining a solution {to} its clearing problem, $(\bm x_m^{\star}, z_m^{\star})$  {(Layer 1), thus after meeting its own flexibility needs, to decide on which remaining (portions of) distributed flexibility bids can be safely forwarded to the TSO-level market.}  {The main idea of the proposed Algorithm \ref{alg:bid_filtering} is to iteratively filter out the costliest bid one by one until the remaining bids are grid safe. } The filtering of the upward and downward bids are done separately albeit done in the same manner. 		
			Based on the outcome of Algorithm \ref{alg:bid_filtering}, i.e., the sets of filtered upward and downward bids denoted by, respectively, $\mc U_m^{\mathrm f}$ and $\mc D_m^{\mathrm f}$, the market clearing problem of Layer 2 in \eqref{eq:layer_tso} is modified by replacing \eqref{eq:limit_u_d_dn_tso} with 
			\begin{align}
				\begin{cases}
					u_{m,n} \in [0, u_{m,n}^{\max}- u_{m,n}^{\star}],\ \forall n \in \mc U_m^{\mathrm f}, \forall m \in \mc N^{\mathrm D}, \\
					d_{m,n} \in [0, d_{m,n}^{\max}- d_{m,n}^{\star}],\ \forall n \in \mc D_m^{\mathrm f}, \forall m \in \mc N^{\mathrm D}, \\
					u_{m,n} = 0, \ \forall n \in \mc U_m \setminus \mc U_m^{\mathrm f}, \forall m \in \mc N^{\mathrm D},  \\
					d_{m,n} = 0, \ \forall n \in \mc D_m \setminus \mc D_m^{\mathrm f}, \forall m \in \mc N^{\mathrm D}.
				\end{cases}  
				\label{eq:limit_u_d_dn_tso_2}
			\end{align}

			\subsection{{Method Analysis}}
			{As shown next,} under some mild assumptions on the bid prices and the structure of the distribution networks, which are typically satisfied in practice\footnote{Assumption \ref{as:bid_prices}, i.e., having more expensive upward bids than downward ones is common in practice, e.g., in balancing markets. Meanwhile, Assumption \ref{as:radial_dn} considers the most common structure of distribution systems.}, we {can} guarantee that by implementing Algorithm \ref{alg:bid_filtering}, the outcome of the sequential market is safe both for the TSO and DSOs. {We also characterize the lower and upper suboptimality bounds of Algorithm \ref{alg:bid_filtering}.} 
			\begin{assumption}
				\label{as:bid_prices}
				For each $m \in \mc N^{\mathrm D}$, the bid prices of FSPs in DSO-$m$ follow that
				$\max_{n \in \mc D_m} c_{m,n}^d < \min_{n \in \mc U_m} c_{m,n}^u.$
			\end{assumption}
			\begin{assumption}
				\label{as:radial_dn}
				For each $m \in \mc N^{\mathrm D}$, the distribution network $m$ is radial.
			\end{assumption}
			\begin{proposition}
				\label{prop:grid_safe_alg1}
				Let Assumptions \ref{as:bid_prices}--\ref{as:radial_dn} hold. Let us consider Problem \ref{eq:layer_tso} where \eqref{eq:limit_u_d_dn_tso} is replaced with \eqref{eq:limit_u_d_dn_tso_2}, in which $\mc U_m^{\mathrm f}$ and $\mc D_m^{\mathrm f}$ are obtained from Algorithm \ref{alg:bid_filtering}. Then, these statements hold:
				\begin{enumerate}[label=\roman*]
					\item \label{prop:grid_safe_alg1_point} The cleared bids $(\bm u_m^{\star}+\bm u_m^{\star\star}, \bm d_m^{\star}+\bm d_m^{\star\star}), \forall m \in \mc N^{\mathrm D},$ and $(\bm u_0^{\star\star},\bm d_0^{\star\star})$ are grid safe.
					\item \label{prop:subopt_alg1_1} {If $\mc U_m^{\mathrm f} = \mc U_m$ and $\mc D_m^{\mathrm f} = \mc D_m$, for all $m \in \mc N^{\mathrm D}$, then the cleared bids are an optimal solution to the \emph{idealized} sequential market  {(Definition \ref{def:ideal_sequential}).}}
					\item \label{prop:subopt_alg1_2} {If $\mc U_m^{\mathrm f} = \emptyset$ and $\mc D_m^{\mathrm f} = \emptyset$, for all $m \in \mc N^{\mathrm D}$, then the cleared bids are an optimal solution to the fragmented market  {(Definition \ref{def:fragmented}).}} 
				\end{enumerate}
			\end{proposition}
						
			Proposition \ref{prop:grid_safe_alg1}.\ref{prop:grid_safe_alg1_point} provides a guarantee on the grid-safe outcome of the bid filtering method. Furthermore, in terms of market efficiency, in the best case, similar to the three-layer market method, the solution is as efficient as that of the idealized sequential market (Proposition \ref{prop:grid_safe_alg1}.\ref{prop:subopt_alg1_1}), and this occurs when all bids are forwarded. In the worst case, when no bid is forwarded, the solution is as efficient as that of the fragmented market (Proposition \ref{prop:grid_safe_alg1}.\ref{prop:subopt_alg1_2}). Therefore, in general, the market efficiency of this approach is between the idealized sequential and the fragmented market schemes.
				\begin{remark}
				\label{rem:computation_comp_alg1}
				Algorithm \ref{alg:bid_filtering} requires each DSO to solve at least $2$ LPs, {implying all remaining bids can be forwarded (the case in Proposition \ref{prop:grid_safe_alg1}.\ref{prop:subopt_alg1_1})} and at most $|\mc U_m|+|\mc D_m|$ LPs, {implying at most only two bids are forwarded or no bids can be forwarded (the case in Proposition \ref{prop:grid_safe_alg1}.\ref{prop:subopt_alg1_2}).} 
			\end{remark}
			
			The computational load of the method is described in Remark \ref{rem:computation_comp_alg1}. Since the upper bound of the computational load is linearly proportional to the number of bids, this can be a practical limitation for networks with a large number of flexibility resources.   	
		
		For transmission-level markets, the EU electricity balancing guideline \cite{eu2017-2195} foresees a role for the TSO of declining the use of certain balancing flexibility due to grid concerns. Extending this ability to DSOs would provide regulatory support for the bid filtering method next to its implementation simplicity. Nonetheless, the process of bid filtering must be transparently implemented and explained, to provide market participants with full transparency on the reasons for disqualifying bids. The absence of this transparency would lead to a low participation potential of FSPs, leading to the potential failure of the associated flexibility market initiative. As DSOs are increasingly adapting their processes of prequalification, such bid filtering aspects can then be integrated in dynamic prequalification processes, with clearly defined rules and conditions.
	
	\section{Bid Aggregation Method}
	\label{sec:bid_aggregation}
	
	The second preventive method is a bid aggregation approach {{originally conceptualized} in the Smartnet project \cite{smartnet} and presented in \cite{papavasiliou2018coordination,papavasiliou2020hierarchical,mezghani2021coordination}.} {We further develop this method and extend its application to our DSO-TSO coordinated setting.}
	
\subsection{{Model Formulation}}
In this method, first, each DSO clears its local market for a number of discrete interface flow values. Then, it constructs a {discretized} {Residual Supply Function} (RSF), which is a one-to-one correspondence between the interface flow values and the optimal costs of Layer 1 for the corresponding fixed interface flow values. Then, the DSOs forward the discrete set of interface flow values {generated,} { $\mc Z_m^{\mathrm{dsc}}$,} and the RSF, {forming a step-wise bidding function, where each step of this function can be safely cleared as it was generated taking into account the distribution grid constraints.} With this knowledge, the TSO clears its market by considering the discrete interface flows (and their RSFs) {-- generated through the aggregation of the distribution bids --} and transmission-level bids. Once a solution to this problem is obtained, the TSO informs each DSO the selected interface flow {step} and, in turn, each DSO selects the optimal distribution-level bid {clearing solution} that corresponds to this {step,} as cleared in the first stage. This method requires modifications on both layers of the sequential market and is summarized in Algorithm \ref{alg:bid_aggregation}. 	

\begin{algorithm}[!t]
	\caption{Bid aggregation method}
	\label{alg:bid_aggregation}
	\begin{enumerate}
		\item Each DSO-$m$, for each $m \in \mc N^{\mathrm D}$, defines the ordered discrete set of feasible interface flow, $\mc Z_m^{\mathrm{dsc}} \subset [z_m^{\min},z_m^{\max}]$, where $|\mc{Z}_m^{\mathrm{dsc}}| =: N_{\mc Z,m} < \infty$.
		\item Each DSO-$m$ clears its local market with fixed interface flow values, i.e. for each $\hat{z}_{m} \in \mc Z_m^{\mathrm{dsc}}$, it solves Problem \eqref{eq:layer_dso} with $c_m^z=0$ in \eqref{eq:cost_dso}, and \eqref{eq:limit_z_dso} being replaced with $z_m = \hat{z}_{m}$. If this problem is infeasible, $\hat z_{m}$ is discarded. Otherwise, it collects the pair $(\hat z_{m,k},J_{m,k})$, where $J_{m,k}$ denotes the optimal cost value. 
		\item Let the discrete set of feasible interface flow values be denoted by $\mc Z_m^{\mathrm{dsc},\star} \subseteq \mc Z_m^{\mathrm{dsc}}$. Then, each DSO-$m$ obtains a residual supply function $ f^{\mathrm{r}}_m \colon \mc Z_m^{\mathrm{dsc},\star} \mapsto \mathbb R:=J_{m,k}(\hat z_{m})$.
		\item The TSO clears its market by solving:
		\begin{subequations}
			\begin{align}
				\min_{\bm x_0, (\bm x_m, z_m)_{m \in \mc N^{\mathrm D}}} \ \ & \bm c_0^\top \bm x_0 + \textstyle\sum_{m \in \mc N^{\mathrm{D}}}f_m^{\mathrm{r}}(z_m) \notag \\
				\operatorname{s.t.} \ \ & \text{\eqref{eq:power_balance_tso_detail1}, \eqref{eq:power_balance_tso_detail2}, \eqref{eq:limit_flow_tso}, \eqref{eq:limit_u_d_tso}, } \notag \\
				& z_m \in \mc Z_m^{\mathrm{dsc},\star}, \ \forall m \in \mc N^{\mathrm D}. \label{eq:limit_z_discrete}
			\end{align}%
			\label{eq:layer_tso_rsf}%
		\end{subequations}%

		\item The TSO informs the optimal interface flow $z_m^{\star\star}$ to each DSO $m \in \mc N^{\mathrm{D}}$ and then each DSO $m$ clear bids based on the solution in Step 2 given $z_m^{\star\star}$.
	\end{enumerate}
\end{algorithm}

Algorithm \ref{alg:bid_aggregation} is different from the RSF-based methods discussed in \cite{papavasiliou2018coordination,papavasiliou2020hierarchical,mezghani2021coordination} mainly in the way the RSFs are constructed. While \cite{papavasiliou2018coordination} considers a continuous linear approximation of the RSF, the papers \cite{papavasiliou2020hierarchical,mezghani2021coordination} use a step-wise price based on the subgradient of the RSFs, i.e., the dual optimal values of the constraint $z_m = \hat z_m$ in the problem solved in Step 2 of Algorithm \ref{alg:bid_aggregation}. Instead, we use the discretized RSFs, which enable controlling the optimality of the market outcome as we show in the next section. 
	The efficiency gains of the proposed method will be showcased in the numerical comparison in Section \ref{sec:num_sim}.	

	\subsection{{Method Analysis}}
	As shown next, we can prove that this bid aggregation method solves a restricted common market clearing problem and, thus, we can obtain a theoretical optimality bound, which is not provided in  \cite{papavasiliou2018coordination,papavasiliou2020hierarchical,mezghani2021coordination}.  {However, let us first} guarantee that the outcome of Algorithm \ref{alg:bid_aggregation} is always grid safe, as formally stated in the next proposition.	
	\begin{proposition}\label{prop:grid_safe_alg2}
The cleared bids obtained by Algorithm \ref{alg:bid_aggregation} are grid safe.
\end{proposition}

Now, we study the efficiency of the solution obtained by Algorithm \ref{alg:bid_aggregation} with respect to the common market solution. {To this end, first we define the RSF step size, {needed for our suboptimality bound.}}
\begin{definition}
\label{def:step_rsf}
For each $m \in \mc N^D$, let us consider the ordered discrete set $\mc Z_m^{\mathrm{dsc}}:=
\{z_{m,1}, z_{m,2}, \dots, z_{m,{N_{\mc Z,m}}} \}$, where $z_{m,1}<\dots< z_{m,{N_{\mc Z,m}}}$. We denote by $\delta_m$ the largest distance between two consecutive elements, i.e., $$\delta_m = \max_{k \in \{2,\dots,{N_{\mc Z,m}}\}} |z_{m,k}-z_{m,k-1} |.$$
\end{definition}

We next show that the outcome of {the bid aggregation method} (Algorithm \ref{alg:bid_aggregation}) is a feasible but possibly suboptimal solution to the common market problem in \eqref{eq:common_market}.
\begin{proposition} 
\label{cor:lower_bound_rsf_optimality}
The optimal value of the common market problem in \eqref{eq:common_market} is a  {tight} lower bound to the optimal value of Problem \eqref{eq:layer_tso_rsf} in Algorithm \ref{alg:bid_aggregation}.
\end{proposition}
Furthermore, this suboptimality of the solution obtained by Algorithm \ref{alg:bid_aggregation} depends on the RSF step size, as shown next.

\begin{theorem}
\label{th:suboptimality_rsf}
The suboptimality of the sequential market with Algorithm \ref{alg:bid_aggregation} with respect to the optimal value of the common market problem in \eqref{eq:common_market}  is $\mc O(\bar \delta)$, where $\bar \delta = \max_{m \in \mc N^D} \delta_m$. 
\end{theorem}

Based on Theorem \ref{th:suboptimality_rsf}, we can infer that the efficiency of the sequential market with the proposed bid aggregation method tends to that of the common market as $\delta_m$ approaches zero. 
{Theorem \ref{th:suboptimality_rsf} also reveals that the optimality of the practical sequential market with the bid aggregation method is not lower bounded by the the optimal value of the idealized sequential market model (Definition \ref{def:ideal_sequential}), implying that the total cost of the former in fact can be better than the latter, especially when the interface flow prices are not optimal. 
	\begin{remark}\label{rem:rsf_complexity}
		Algorithm 2 requires DSO $m$ to solve $N_{\mc Z, m}$ LPs where $N_{\mc Z, m}$ is inversely proportional to $\delta_m$, i.e. $N_{\mc Z, m}= (z_m^{\max}\hspace{-2pt}-\hspace{-2pt}z_m^{\min})/\delta_m$. Moreover, the TSO must solve a mixed-integer LP instead of an LP as in the other methods.  {These aspects negatively impact the computational load of this method.}
	\end{remark}
	
\begin{remark}
\label{rem:outer_loop_alg2}
We can improve the solution quality of the proposed bid aggregation method by introducing an outer loop to Algorithm \ref{alg:bid_aggregation}. Specifically, given the partial solution $\hat{ \bm z}$, we can  discretize the range $[\hat z_m-\delta_m, \hat z_m + \delta_m]$ to obtain a new discrete set $\hat{\mc Z}_m^{\mathrm{dsc}}$ and redo steps 2--4 of Algorithm \ref{alg:bid_aggregation}. Thus, we practically iteratively reduce $\delta_m$, which can be more tractable than simply performing Algorithm \ref{alg:bid_aggregation} with a very small $\delta_m$.
\end{remark}

{The efficiency and grid safety guarantees of this method supports the potential of its practical implementation. However, a regulatory challenge remains in terms of the role of the entity that would be responsible for performing this bid aggregation, especially when this role is to be assumed by the DSO (given the need for the knowledge of the grid model for the computation of the step-wise offer curve and RSF). Finally, this method is computationally demanding, both for the DSOs and the TSOs, as demonstrated by Remark \ref{rem:rsf_complexity}. The complexity of the mixed-integer problem in Step 4 of Algorithm \ref{alg:bid_aggregation} scales with the number of distribution networks (as the method is applied for each single transmission-distribution connection), which is related to the dimension of the integer decision variables, implying a potential limitation for the application to large networks. As such, market design rules related to flexibility products and procurement timings should take this complexity into account for the successful implementation of this method.

\section{Numerical Simulations}
\label{sec:num_sim}
 
\begin{figure}
\centering
\includegraphics[width=0.9\columnwidth]{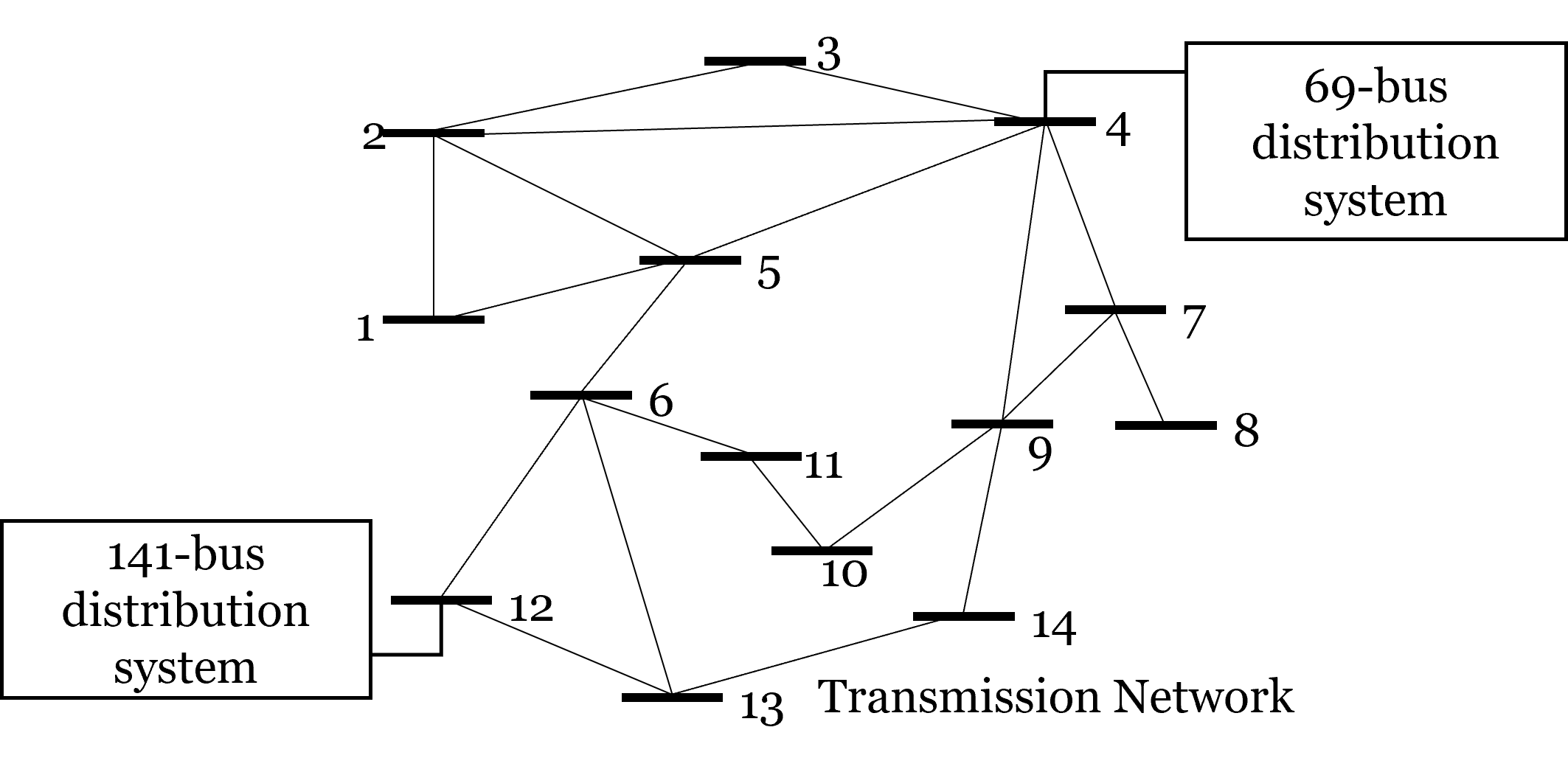}
\caption{ The network topology for the simulation study. }
\label{fig:network_structure}
\end{figure}   
In this section, we {present a case study to numerically compare} the performance of {proposed methods} in terms of their solution feasibility, efficiency, {and computational time.} 
To this end, we use the IEEE 14-bus transmission network interconnected with the Matpower 69-bus and 141-bus radial distribution networks, as shown in Fig. \ref{fig:network_structure}. 
We study four cases, whose datasets are available in \cite{ananduta2023dataset}.  In all cases, base injections and loads of the busses are adjusted to create {power imbalances, i.e., where the load and generation do not match,} while the line limits are also modified to create anticipated congestion in the networks. Moreover, upward and downward flexibility bids are randomly allocated to the busses. {The distinguishing features of each case are given as follows:} 

\begin{itemize}[leftmargin=*]
\item Case A: The {TSO} has an upward flexibility need {(power imbalances).} The prices of the downward bids are in the range [10, 25] €/MW while those of the upward bids are in the range [30, 55] €/MW, satisfying Assumption \ref{as:bid_prices}. Additionally, the distribution-level bids are more expensive than the transmission-level ones. 
\item Case B: The networks have imbalances as in Case A. However, the transmission-level upward bids, in the range $[90, 165]$ €/MW, are more expensive than the distribution-level ones. The prices of the other bids are as in Case A.
\item Case C: The networks have imbalances, as in Case A. The transmission-level upward bids are more expensive than the distribution-level ones as in Case B. We add new upward distribution-level bids whose prices are more expensive than those in the transmission network and new downward distribution-level bids whose prices are as in Case B. 
\item Case D: The {TSO} has a downward need. The price rules are set the same as Case A. 
\end{itemize}

For the three-layer and bid filtering method, we apply the interface flow price rules in \cite[Sect. III]{marques2023grid}, namely:  1) no pricing, i.e., when they are not priced ($c_m^z = 0, \forall m \in \mc N^{\mathrm D}$); 2) optimal, i.e., when they are priced by the optimal dual variables of the power balance constraints  {\eqref{eq:power_balance_tso_detail1}} of the common market problem (see Proposition \ref{prop:optimal_price}); {and 3) midpoint, i.e., when the price is the average of the most expensive downward bid and the least expensive upward bid.} In our simulation, we simply obtain the optimal prices by solving Problem \eqref{eq:common_market}, but this is not realistic in practice, {as explained in Section \ref{sec:grid_safety_considerations}.} 
Note that in the bid aggregation method, the cost of using the interface flows is determined by the RSFs  {in which 10 different step-size values of $\bar \delta$ are used.} 

As previously mentioned, we use the optimal cost of the common market problem in \eqref{eq:common_market}, denoted by $J^{\mathrm{com}}=J^{\mathrm{tot}}(\{\bm x_m^\circ\}_{m \in \mc N})$, as the benchmark to determine the inefficiency of the sequential market when paired with each of the proposed methods, defined as:
\begin{equation*} 
\eta = \frac{J^{\mathrm{tot}}-J^{\mathrm{com}}}{|J^{\mathrm{com}}|} \times 100\%.
\end{equation*}

We summarize our case study results in Table \ref{tab:sim_res_eff}.  
The bid filtering and bid aggregation methods obtain grid-safe cleared bids in {all} cases, as proven in Propositions \ref{prop:grid_safe_alg1} and \ref{prop:grid_safe_alg2}. On the other hand, the three-layer corrective method produces unsafe cleared bids in Case B, due to {the} procurement of distribution-level bids in Layer 2 that causes unresolvable congestion. When the system is liquid enough, as in Case C, where we add upward and downward bids in critical nodes (whose lines are congested in the {results} of the three-layer market in Case B), we observe that this scheme obtain grid-safe bids.

In terms of efficiency, we observe that the performance of these methods are case-dependent. In Cases A and D, the three-layer and bid filtering methods {with the same interface pricing} obtain equal solutions, {as} distribution-level bids are more expensive than the transmission-level ones, and thus the {distribution-level bids} are not cleared in the second layer while their first-layer solutions are equal. Meanwhile, in Case C, we can observe that the bid filtering method {under either the no-pricing or midpoint rule} obtains the most inefficient solutions among the three methods, although the three-layer market scheme incurs an additional cost from resolving congestion in the third layer. From Case C, we also observe the benefit of forwarding bids as the fragmented market solution is the least efficient under the same pricing rule,  {illustrating Proposition \ref{prop:grid_safe_alg1}.\ref{prop:subopt_alg1_2}.}   On the other hand, the bid aggregation method {generally} achieves low inefficiency, {while} its inefficiency has a decreasing trend as the step size decreases, as proven in Theorem \ref{th:suboptimality_rsf} and evidently shown in the top plot of Fig. \ref{fig:eff_rsf}.

These numerical results also provide an illustration of Proposition \ref{prop:optimal_price}. In Cases A and D, for the three-layer method, when the interface flows are priced optimally (the third row), the second layer market does not clear any distribution-layer bids. In turn, the outcome of the sequential market is as efficient as the common market. However, in Cases B and C, the second-layer market clears some distribution-level bids, thus even under an optimal interface flow price, the outcomes are then not as optimal as the common market. In fact, in Case B, the three layer market obtains an infeasible solution under an optimal interface flow price, due to the low liquidity at the distribution systems. Similarly, for the bid filtering method under the optimal pricing rule, the outcome is as efficient as the common market solution in Cases A and D, as no distribution-level bids are cleared in Layer 2, while in Cases B and C, some distribution-level bids are cleared, thus they cannot be a solution to the common market problem.

The average computational time over all the simulated cases of the three-layer and bid filtering methods are $6.10$ and $126.57$ seconds, respectively. 
The bottleneck of the bid filtering method is indeed in the filtering process, especially when the worst-case scenario, as stated in Remark \ref{rem:computation_comp_alg1}, occurs. On the other hand, the computational time of the bid aggregation method on the simulated cases varies in the range of $[8.75, 135.75]$ seconds and i} depends on the step size of the RSF, $\bar \delta$, which determines the number of LPs that must be solved and the dimension of the set of binary variables required in the TSO problem. As we observe in the bottom plot of Fig. \ref{fig:eff_rsf}, the computational time requirement of the bid aggregation method grows exponentially as $\bar \delta$ decreases. From Fig. \ref{fig:eff_rsf}, we can indeed infer that there is a trade-off between solution quality and computational time.  The computational time results indicate the applicability potential of these methods even for near real-time markets. 
In addition, we note that the computational time of each method depends on the size of the distribution system as the number of constraints of the LPs that must be solved in any of the methods is linearly proportional to the network size.  Nevertheless, any LP can be solved in polynomial time and efficiently using existing off-the-shelf solvers.

\begin{table}
\caption{Inefficiency comparison}
\label{tab:sim_res_eff}
\begin{tabular}{|l|cccc|}
	\hline
	
	\multirow{2}{*}{Method} & \multicolumn{4}{c|}{Inefficiency $\eta$ $[\%]$} \\
	
	& Case A & Case B & Case C & Case D  \\
	\hline
	Three-layer$^1$	& 4.35  & {infeasible}  & \cellcolor{yellow}\textcolor{blue}{22.50} & 237.67 \\
	Three-layer$^2$	& 0.0 & {infeasible} & \cellcolor{yellow}\textcolor{blue}{7.17} & 0.0  \\
	Three-layer$^3$	& 0.0  & {infeasible}  & \cellcolor{yellow}\textcolor{blue}{10.36} & 0.0 \\
	Bid filtering$^1$	& 4.35  & 31.67  & 31.67 & 31.67  \\
	Bid filtering$^2$	& 0.0 & 0.0 & 0.0 & 0.0 \\
	Bid filtering$^3$	& 0.0  & 17.03  & 17.03 & 0.0 \\
	Bid-aggregation$^4$	& \hspace{-5pt} [0.04, 1.8] &\hspace{-5pt} [0.03, 5.2] &\hspace{-5pt} [0.09, 3.6] & \hspace{-5pt} [1.7, 77.3]  \\
	Fragmented$^3$&  0.0 & 59.17 & 59.17 & 0.0  \\
	\hline 
\end{tabular}

$1$: no pricing; $2$: optimal; $3$: midpoint; $4$: $\bar \delta \in [0.03, 1.3]$.\\
\textcolor{blue}{\hl{blue}}: cleared bids cause congestions but resolvable using the third layer.
\end{table}

\begin{figure}[!tbp]
\centering
\includegraphics[width=0.85\columnwidth]{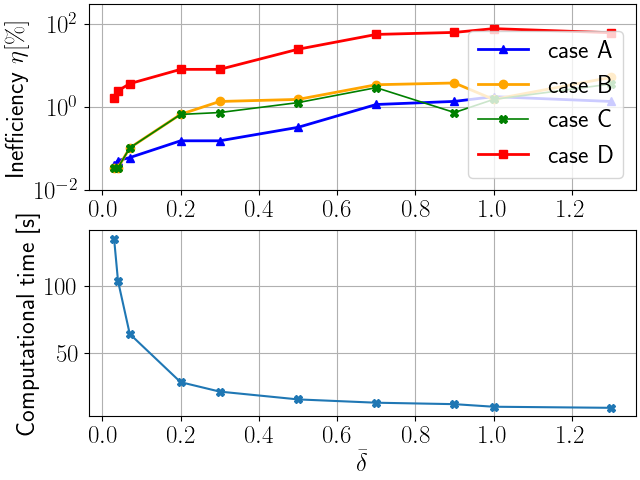}
\caption{The inefficiency of the bid aggregation method {(top) and the average computational time (bottom)} with varying RSF step size $\bar \delta$. }
\label{fig:eff_rsf}
\end{figure} 

Finally, Fig. \ref{fig:price-basedRSF} {compares} the inefficiency {variation} of the bid aggregation method that is based on the primal cost and the dual price RSFs \cite{papavasiliou2020hierarchical,mezghani2021coordination}. {We can observe that the proposed RSF method {outperforms} the dual-price RSF one as it always {achieves a} lower inefficiency for any step size $\bar \delta$.} Furthermore, for the dual-price RSF method, we cannot clearly observe {a decreasing inefficiency trend with a decrease in $\bar{\delta}$.} 

\begin{figure}[!tbp]
\centering
\includegraphics[width=0.85\columnwidth]{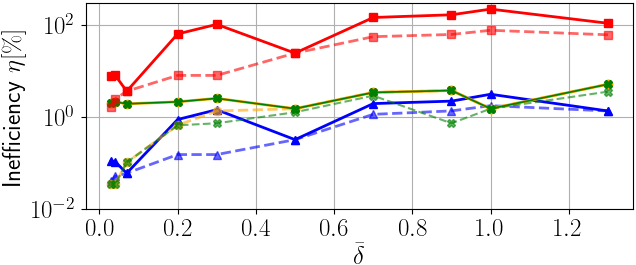}
\caption{The inefficiency of the bid aggregation method with dual-price-based RSFs  \cite{papavasiliou2020hierarchical,mezghani2021coordination} (solid lines) compared with the {proposed} primal-{cost-based} RSFs (dashed lines). {Each line represents a case with the markers and colors following the top plot of Fig. \ref{fig:eff_rsf}.} }
\label{fig:price-basedRSF}
\end{figure}

\section{Conclusion}
\label{sec:conclusion}

\begin{table}[t]
\hspace{-1pt}	{ 
	\caption{Comparison of properties and performances}
	\label{tab:comparison}
	\centering
	\begin{tabular}{|l|c|c|c|}
		\hline
		Metric	& \hspace{-1pt}Three-layer\hspace{-1pt} & Bid filtering & \hspace{-1pt}Bid aggregation\hspace{-1pt}  \\
		\hline 
		Grid-safety	& no & yes, under  & yes\\ 
		guarantee & & Assumptions \ref{as:bid_prices}--\ref{as:radial_dn} & \\
		\hline 
		Inefficiency & highest & middle & lowest \\
		&   & {(bounded)} & (controllable)  \\
		\hline
		Computation load & lowest & middle & highest \\ \hline 
	\end{tabular}
}
\end{table}

In sequential TSO-DSO flexibility markets, when the TSO-layer market does not have sufficient information on the distribution networks, forwarding bids from the DSO-layer markets to the TSO-layer one can result in congestion in distribution systems if not handled carefully. Three methods, namely a three-layer-market scheme, bid filtering, and bid aggregation, can be used to achieve a grid-safe use of distributed flexibility by the TSO.  The theoretical properties and numerical performances of these methods obtained in this work are summarized in Table \ref{tab:comparison}. Although bid aggregation provides the most desired outcome as it can provide a grid-safe guarantee under the most relaxed assumptions and a high efficiency, it can be the most computationally demanding method, hindering its practical implementation potential. While the three-layer-market scheme requires the market to be sufficiently liquid, it can be more efficient than bid filtering; however, this result can be case-dependent. On the other hand, bid filtering is guaranteed to produce grid-safe cleared bids in radial networks and mild assumptions on the bid prices. Analyzing the performance of these methods in handling other grid-unsafe events, such as voltage violations, and under different power flow models, is an important future research avenue. 

\appendix
\label{appendix}
{ 
\subsection{Proof of Proposition \ref{prop:sequential_frag}}
\label{pf:prop:sequential_frag}
The problems solved in Layer 1 of the fragmented market model and the idealized sequential market model are equal. Thus, they have an equal optimal cost value. However,
in the fragmented market model, no distribution-layer bids are cleared in Layer 2 (see Definition \ref{def:fragmented}). Thus, $\bm u_m^{\mathrm f}$ and $\bm d_m^{\mathrm f}$, for each $m \in \mc N^{\mathrm D}$ are obtained by solving Problem \eqref{eq:layer_dso} only. Furthermore, as a consequence of the equality constraint \eqref{eq:aggregated_balance_dn_tso}, the interface flow solution from Layer 2 is the same as that of Layer 1, i.e., $z_m^{\mathrm f} = z_m^{\star\star}=z_m^\star$, {capturing that} the interface flow values are fixed based on the market clearing in Layer 1. Therefore, we can conclude that $\bm u_m^{\mathrm f}$ and $\bm d_m^{\mathrm f}$  satisfy \eqref{eq:power_balance_dso_detail1}, \eqref{eq:power_balance_dso_detail2}, \eqref{eq:limit_flow_dso}, and \eqref{eq:limit_z_dso}. Consequently, $(\bm x_0^{\mathrm f}, (\bm x_m^{\mathrm f}, z_m^f)_{m \in \mc N^{\mathrm D}})$ is a feasible point to the market clearing problem in Layer 2 of the idealized sequential market model, which includes \eqref{eq:power_balance_dso_detail1}, \eqref{eq:power_balance_dso_detail2}, and \eqref{eq:limit_flow_dso} (see Definition \ref{def:ideal_sequential}). However, it is not necessarily an optimal one, and hence, the inequality \eqref{eq:ineq_sequential_frag} {readily applies.} \qedd
}

\subsection{Proof of Proposition \ref{prop:optimal_price}}
\label{pf:prop:optimal_price}

{	 {If the interface flows are priced optimally,} the solution to Problem \eqref{eq:common_market} is equal to that of the fragmented market clearing problem  {(Definition \ref{def:fragmented}) \cite[Prop. 2]{marques2023grid}.} Therefore, if $(\bm u_m^{\star}\hspace{-1pt}+\hspace{-1pt}\bm u_m^{\star\star}, \bm d_m^{\star}\hspace{-1pt}+\hspace{-1pt}\bm d_m^{\star\star}), \forall m \in \mc N^{\mathrm D},$ and $(\bm u_0^{\star\star},\bm d_0^{\star\star})$ are a solution to Problem \eqref{eq:common_market}, then by construction of the fragmented market model,  {where no distribution-level bids are cleared in Layer 2,} it must hold that  $\bm u_m^{\star\star}\hspace{-2pt}=\hspace{-2pt}0$ and $\bm d_m^{\star\star}\hspace{-2pt}=\hspace{-2pt}0, \forall m \in \mc N^{\mathrm D}$. 
} 
\qedd 

\subsection{Proof of Proposition \ref{prop:grid_safe_alg1}}
\label{pf:prop:grid_safe_alg1}
{We first prove Proposition \ref{prop:grid_safe_alg1}.\ref{prop:grid_safe_alg1_point}.} The cleared transmission-layer bids, $(\bm u_0^{\star\star},\bm d_0^{\star\star})$, satisfy  {\eqref{eq:power_balance_tso_detail1}, \eqref{eq:power_balance_tso_detail2},} and \eqref{eq:limit_flow_tso} by construction, as they are obtained by solving Problem \eqref{eq:layer_tso}. Now we show that the cleared distribution-level bids $(\bm u_m^{\star}+\bm u_m^{\star\star}, \bm d_m^{\star}+\bm d_m^{\star\star}), \forall m \in \mc N^{\mathrm D},$ respect the distribution network constraints even though $(\bm u_m^{\star\star}, \bm d_m^{\star\star}), \forall m \in \mc N^{\mathrm D},$ are obtained from Layer 2, which does not include such constraints. 

To that end, we need the following lemma.
\begin{lemma}
	\label{le:either_upward_downward_cleared}
	Let Assumption \ref{as:bid_prices} hold. Let $(\bm u_m^{\star\star}, \bm d_m^{\star\star})_{m \in \mc N^{\mathrm D}}$ be the cleared distribution-level bids obtained by solving Problem \eqref{eq:layer_tso} (Layer 2). Then, for each $m \in \mc N^{\mathrm D}$, only one of the following conditions holds:
	\begin{enumerate}
		\item No downward bids are cleared, i.e., $\bm d_m^{\star\star} = 0$.
		\item No upward bids are cleared, i.e., $\bm u_m^{\star\star} = 0$. 
		\item Both upward and downward bids are not cleared, i.e., $\bm u_m^{\star\star} = 0$ and $\bm d_m^{\star\star} = 0$.
	\end{enumerate}
\end{lemma}
\begin{IEEEproof}[Proof of Lemma \ref{le:either_upward_downward_cleared}]
	For each $m \in \mc N^{\mathrm D}$,we denote the net flexibility position by $\omega_m = \1^\top \bm u_m^{\star\star} - \1^\top \bm d_m^{\star\star}$. By Assumption \ref{as:bid_prices}, it can be shown {by contradiction} that condition 1 holds for any $\omega_m > 0$, condition 2 holds for any $\omega_m < 0$, and condition 3 holds for $\omega_m = 0$, implying that all possible values of $\omega_m$ {are covered. We}  show one of the cases, i.e., $\omega_m > 0$, as the {proofs of the} other {conditions and the corresponding cases} follow the same lines of reasoning. 
	
	For the sake of contradicting {condition 1,} suppose that, {when  $\omega_m > 0$,}  some downward bids are cleared, i.e., $\1^\top \bm d_m^{\star\star} > 0$. Then, $\1^\top \bm u_m^{\star\star} = \omega_m + \1^\top \bm d_m^{\star\star} > \omega_m$. Let us now consider another set of bids $(\tilde{\bm u}_m, \tilde{\bm d}_m)$, where $\1^\top \tilde{\bm u}_m = \omega_m$ and $\tilde{\bm d}_m=0$. This bid set is a feasible solution to Problem \eqref{eq:layer_tso} since it satisfies all the constraints. Furthermore, the cost difference between $(\bm u_m^{\star\star}\hspace{-2pt},\hspace{-2pt} \bm d_m^{\star\star})$ and $(\tilde{\bm u}_m\hspace{-2pt},\hspace{-2pt} \tilde{\bm d}_m)$ can be written as:
	{
		\begin{align*} 
			&\bm c_m^{u\top} (\bm u_m^{\star\star}-\tilde{\bm u}_m) -\bm c_m^{d\top} (\bm d_m^{\star\star}-\tilde{\bm d}_m)	\\
			&	\stackrel{(a)}{\geq} (\min_{m \in \mc U_m} c_{m,n}^u)\1^{\top} (\bm u_m^{\star\star}-\tilde{\bm u}_m) - (\max_{m \in \mc D_m} c_{m,n}^d)\1^{\top} \bm d_m^{\star\star} \\
			& \stackrel{(b)}{=} (\min_{m \in \mc U_m} c_{m,n}^u - \max_{m \in \mc D_m} c_{m,n}^d)\1^{\top} \bm d_m^{\star\star} 
			\stackrel{(c)}{>} 0,
	\end{align*}}%
	where $(a)$ holds because {$\tilde{\bm d}_m=0$,} $\1^\top \bm u_m^{\star\star} > \omega_m =\1^\top \tilde{\bm u}_m$, and $\1^\top \bm d_m^{\star\star} > 0$; $(b)$ holds because $\omega_m \textbf{}=\textbf{} \1^\top \bm u_m^{\star\star} \hspace{-2pt}-\hspace{-2pt} \1^\top \bm d_m^{\star\star}$ and $\omega_m \hspace{-2pt}=\hspace{-2pt}\1^\top \tilde{\bm u}_m$; {and} $(c)$ holds due to Assumption \ref{as:bid_prices}. Hence, $(\tilde{\bm u}_m, \tilde{\bm d}_m)$ is cheaper, { implying that $(\bm u_m^{\star\star}, \bm d_m^{\star\star})$, is not an optimal solution and should not have been cleared,}  (a contradiction).  
	\end{IEEEproof}	
	
	\medskip 
	Now, we proceed with the proof of Proposition \ref{prop:grid_safe_alg1}.\ref{prop:grid_safe_alg1_point}. As a result of Algorithm \ref{alg:bid_filtering}, if, in Layer 2, all the forwarded upward bids in $\mc U_m^{\mathrm f}$ are fully cleared while all the downward bids in $\mc D_m^{\mathrm f}$ are rejected, then the cleared bids do not violate the grid constraints  {\eqref{eq:power_balance_dso_detail1}, \eqref{eq:power_balance_dso_detail2},} \eqref{eq:limit_flow_dso}, and \eqref{eq:limit_z_dso}. The same implication holds when all the forwarded downward bids in $\mc D_m^{\mathrm f}$ are fully cleared while all the forwarded {upward} bids in $\mc U_m^{\mathrm f}$ are rejected. Furthermore, by Assumption \ref{as:radial_dn}, $C_m$ consists of non-positive elements only  {(as can be defined through the sign convention of the PTDF matrix of a radial system).} Due to this fact and the linear relationship between $\bm u_m$, $\bm d_m$, and $\bm p_m$ in  {\eqref{eq:power_balance_dso_detail1}--\eqref{eq:power_balance_dso_detail2},} the extreme values of $C_m p_m$ in \eqref{eq:limit_flow_dso} occur when either  $\bm u_m = \bm u_m^{\max}$ and $\bm d_m = 0$ or $\bm u_m = 0$ and $\bm d_m = \bm d_m^{\max}$. Therefore, for any $\bm u_m^{\star\star} \leq \bm u_m^{\max} - \bm u_m^\star$ while $\bm d_m = 0$ or for any $\bm d_m^{\star\star} < \bm d_m^{\max} - \bm d_m^\star$ while $\bm u_m = 0$, the cleared distribution-level bids, $(\bm u_m^{\star}+\bm u_m^{\star\star}, \bm d_m^{\star}+\bm d_m^{\star\star}), \forall m \in \mc N^{\mathrm D},$ are feasible, i.e., they satisfy  {\eqref{eq:power_balance_dso_detail1}, \eqref{eq:power_balance_dso_detail2},} \eqref{eq:limit_flow_dso},  and \eqref{eq:limit_z_dso}. Even though both $\mc U_m^{\mathrm f}$ and $\mc D_m^{\mathrm f}$ are forwarded to Layer 2, 
	Lemma \ref{le:either_upward_downward_cleared} ensures that they cannot be cleared simultaneously.  
	
	Proposition  \ref{prop:grid_safe_alg1}.\ref{prop:subopt_alg1_1} holds by the fact that if all bids are forwarded, then the set of bids considered in Problem \eqref{eq:layer_tso} and that in Layer 2 of the idealized sequential market coincide. Since the cleared bids are grid-safe, then $(u_m^{\star\star}, d_m^{\star\star})$, for all $m \in \mc N$, must also be a solution to Layer 2 of the idealized sequential market. 
Proposition \ref{prop:grid_safe_alg1}.\ref{prop:subopt_alg1_2} holds by the construction of the constraints in \eqref{eq:limit_u_d_dn_tso_2}, given that $\mc U_m^{\mathrm f} = \emptyset$ and $\mc D_m^{\mathrm f} = \emptyset$, for all $m \in \mc N^{\mathrm D}$, which implies the equivalence  with the fragmented market model (Definition \ref{def:fragmented}).
\qedd %

\subsection{Proof of Proposition \ref{prop:grid_safe_alg2}}
\label{pf:prop:grid_safe_alg2}

In Problem \eqref{eq:layer_tso_rsf} of Algorithm \ref{alg:bid_aggregation}, the grid constraints of the transmission network  {\eqref{eq:power_balance_tso_detail1}, \eqref{eq:power_balance_tso_detail2}}, and \eqref{eq:limit_flow_tso} are included. Furthermore, $ \mc Z_m^{\mathrm{dsc},\star} \subset [z_m^{\min},z_m^{\max}]$. Finally, as the distribution-level bids that are cleared are obtained as a solution to Problem \eqref{eq:layer_dso} with a fixed interface flow value. Then, these bids are safe for their distribution network. 
\qedd

\subsection{Intermediate results in Section \ref{sec:bid_aggregation}}
\label{sec:intermediate_results}
To prove Proposition \ref{cor:lower_bound_rsf_optimality} and Theorem \ref{th:suboptimality_rsf}, we need the following intermediate results. 
\begin{lemma}
\label{le:rsf_solves_MILP}
Algorithm \ref{alg:bid_aggregation} solves 
\begin{equation}
	\begin{aligned}
		& 	\min_{\bm x_0, (\bm x_m, z_m)_{m \in \mc N^{\mathrm D}}} \ \  \bm c_0^\top \bm x_0 + \textstyle\sum_{m\in \mc N^D} \bm c_m^\top \bm x_m \\
		& \operatorname{s.t.} \  \text{ {\eqref{eq:power_balance_tso_detail1}, \eqref{eq:power_balance_tso_detail2},} \eqref{eq:limit_flow_tso}, \eqref{eq:limit_u_d_tso}, }\\
		& \qquad \ \text{ {\eqref{eq:power_balance_dso_detail1}, \eqref{eq:power_balance_dso_detail2},} \eqref{eq:limit_flow_dso}, \eqref{eq:limit_u_d_dso},  } z_m \in \mc Z_m^{\mathrm{dsc}},  \forall m \hspace{-2pt}\in\hspace{-2pt} \mc N^{\mathrm D}.
	\end{aligned}
	\label{eq:rsf_lemma}
\end{equation}
\end{lemma}
\begin{IEEEproof}
Let us consider the following optimization:
\begin{equation}
	\begin{aligned}
		&	\min_{\bm x_0, (\bm x_m, z_m)_{m \in \mc N^{\mathrm D}}} \ \   \bm c_0^\top \bm x_0 + \textstyle\sum_{m\in \mc N^D} \bm c_m^\top \bm x_m \\
		&	\operatorname{s.t.} \ \  \text{ {\eqref{eq:power_balance_tso_detail1}, \eqref{eq:power_balance_tso_detail2},} \eqref{eq:limit_flow_tso}, \eqref{eq:limit_u_d_tso}, }\\
		& \qquad \ \ \text{ {\eqref{eq:power_balance_dso_detail1}, \eqref{eq:power_balance_dso_detail2},} \eqref{eq:limit_flow_dso}, and \eqref{eq:limit_u_d_dso},  } \forall m \in \mc N^{\mathrm D}, \\
		& \qquad \ \ z_m = \hat z_{m}, \forall m \in \mc N^{\mathrm D},
	\end{aligned}
	\label{eq:rsf_lemma2}
\end{equation}
for each $\hat z_{m} \in \mc Z_m^{\mathrm{dsc},\star}$ and all $m \in \mc N^{\mathrm{D}}$. One can solve Problem \eqref{eq:rsf_lemma} by enumerating the solutions to Problem \eqref{eq:rsf_lemma2}, for all $\hat z_{m} \in \mc Z_m^{\mathrm{dsc},\star}$ and $m \in \mc N^{\mathrm{D}}$. Note that, by definition of $\mc Z_m^{\mathrm{dsc},\star}$, for any $\hat z_{m} \in \mc Z_m^{\mathrm{dsc}}\setminus\mc Z_m^{\mathrm{dsc},\star}$, Problem \eqref{eq:rsf_lemma2} is infeasible.  We observe that \eqref{eq:rsf_lemma2} is decomposable, i.e., it is equivalent to:

{
	\begin{align} 
		&	\begin{cases}
			\min_{\bm x_0, \bm z}  &  \bm c_0^\top \bm x_0 \\
			\operatorname{s.t.}  & \text{ {\eqref{eq:power_balance_tso_detail1}, \eqref{eq:power_balance_tso_detail2},} \eqref{eq:limit_flow_tso}, \eqref{eq:limit_u_d_tso}, } \\
			& z_m = \hat z_{m}, \forall m \in \mc N^{\mathrm D}
		\end{cases}\\
		&+ \sum_{m \in \mc N^{\mathrm{D}}} 	\begin{cases}
			\min_{\bm x_m, z_m}  &   \bm c_m^\top \bm x_m \\
			\operatorname{s.t.} 
			& \text{ {\eqref{eq:power_balance_dso_detail1}, \eqref{eq:power_balance_dso_detail2},} \eqref{eq:limit_flow_dso}, and \eqref{eq:limit_u_d_dso},  }  \\
			& z_m = \hat z_{m}.
		\end{cases} \label{eq:dso_problem_rsf}
\end{align}}
\indent Notice that Problem \eqref{eq:dso_problem_rsf} is solved in Step 2 of Algorithm \ref{alg:bid_aggregation}, for each $\hat z_{m,k} \in \mc Z_m^{\mathrm{dsc},\star}$, by each DSO $m \in \mc N^{\mathrm{D}}$. The residual function $f_m^{\mathrm r}$ is a step function whose graph is defined by the pairs $(\hat z_{m},J_{m})$, for all $\hat z_{m} \in \mc Z^{\mathrm{dsc},\star}$,  where $J_{m}$ is the optimal value of Problem \eqref{eq:dso_problem_rsf}, i.e. if $\hat{\bm x}_m$ denotes a solution to Problem \eqref{eq:dso_problem_rsf}, then  $J_m := \bm c_m^\top \hat{\bm x}_m$. Therefore, Problem \eqref{eq:rsf_lemma2} is equivalent to
\begin{equation}
	\begin{aligned}
		\min_{\bm x_0, (\bm x_m, z_m)_{{m \in \mc N^{\mathrm D}}}} \ \ &  \bm c_0^\top \bm x_0 + \textstyle\sum_{m\in \mc N^D} f_m^{\mathrm r}(\hat z_m) \\
		\operatorname{s.t.} \ \ & \text{ {\eqref{eq:power_balance_tso_detail1}, \eqref{eq:power_balance_tso_detail2},} \eqref{eq:limit_flow_tso}, \eqref{eq:limit_u_d_tso}.}
	\end{aligned}
	\label{eq:rsf_lemma3}
\end{equation}
\indent Consequently, solving Problem \eqref{eq:layer_tso_rsf} in Step 4 of Algorithm \ref{alg:bid_aggregation} is equivalent to enumerating the solutions to Problem \eqref{eq:rsf_lemma3}, for all feasible interface flow values, implying the equivalence of Problems \eqref{eq:layer_tso_rsf} and \eqref{eq:rsf_lemma}.
\end{IEEEproof}

\begin{lemma}
\label{le:distance_rsf_common_sols}
Let $(\hat{ \bm x}, \hat{ \bm z})$ be the solution computed by Algorithm \ref{alg:bid_aggregation}. Then, there exists   a solution to the common market problem \eqref{eq:common_market}, denoted by $(\bm x^\circ, \bm z^\circ)$ such that
\begin{equation}
	|  { \hat z_m} -  z_m^\circ | \leq \delta_m, \quad \forall m \in \mc N^{\mathrm D}, \label{eq:distance_rsf_common_sols}
\end{equation}
where $\delta_m$ is as in Definition \ref{def:step_rsf}.
\end{lemma}
\begin{IEEEproof}
By Lemma \ref{le:rsf_solves_MILP}, Algorithm \ref{alg:bid_aggregation} solves Problem \eqref{eq:rsf_lemma}, which has the same cost function and constraints as the common market problem \eqref{eq:common_market} except that $\bm z$ is constrained by the discrete set in \eqref{eq:limit_z_discrete}, which is more restricted than \eqref{eq:limit_z_tso}. Furthermore, by the linear equality constraints in \eqref{eq:power_balance_dso}, for all $m \in \mc N^{\mathrm D}$ and  {in \eqref{eq:power_balance_tso_detail1}--\eqref{eq:power_balance_tso_detail2}, which can be compactly represented as 
	\begin{equation}
		A_0 \bm x_0 + B_0 \bm z = \bm e_0,  \label{eq:power_balance_tso}
	\end{equation}
	with appropriate matrices $A_0$, which has full rank, and $B_0$,}
the cost function \eqref{eq:common_market_cost} can be written as 
\begin{align*}
	J^{\mathrm{tot}}&=\bm c_0^\top \bm x_0 + \textstyle\sum_{m\in \mc N^D} \bm c_m^\top \bm x_m\\
	&	=-(c_0^\top A_m^{\dagger}B_0 + (\col((c_m^\top A_m^{\dagger}B_m)_{m \in \mc N^{\mathrm D}}))^\top)\bm z \\
	&\quad 	+ c_0^\top A_0^{\dagger}\bm e_0 + \textstyle\sum_{m \in \mc N^{\mathrm D}}c_m^\top A_m^{\dagger}\bm e_m,
\end{align*}
where $(\cdot)^{\dagger}$ denotes the pseudo-inverse operator. Thus, $J^{\mathrm{tot}}$ is linearly proportional to $\bm z$.  Since the optimal cost of Problem \eqref{eq:common_market} is a lower bound to the cost of Problem \eqref{eq:rsf_lemma}, $\hat z_m$, for each $m \in \mc N^{\mathrm D}$, is an element in $\mc Z_m^{\mathrm{dsc}}$ closest to a (partial) solution to Problem \eqref{eq:common_market}, $z_m^\circ$. Since in $\mc Z_m^{\mathrm{dsc}}$ the distance between two consecutive elements is at most $\delta_m$, \eqref{eq:distance_rsf_common_sols} must hold.  
\end{IEEEproof}
\subsection{Proof of Proposition \ref{cor:lower_bound_rsf_optimality}}
\label{pf:cor:lower_bound_rsf_optimality}
By Lemma \ref{le:rsf_solves_MILP}, Algorithm \ref{alg:bid_aggregation} solves Problem \eqref{eq:rsf_lemma} where $\mc Z^{\mathrm{dsc}} \subsetneq \mc Z$. Therefore, the common market problem in \eqref{eq:common_market} is a convex relaxation to Problem \eqref{eq:rsf_lemma}, implying that the optimal value of the former is a lower bound of the latter.  This bound is tight, i.e., the optimal costs of Problems \eqref{eq:common_market} and \eqref{eq:rsf_lemma} are equal when optimal interface flow values of Problem \eqref{eq:common_market} lie in the discrete sets $\mc Z_m^{\mathrm{dsc}}$, for all $m \in \mc N^{\mathrm D}$, i.e., $z_m^\circ \in \mc Z_m^{\mathrm{dsc}}$.  \qedd

\subsection{Proof of Theorem \ref{th:suboptimality_rsf}}
\label{pf:th:suboptimality_rsf}
Let $(\hat{ \bm x}, \hat{ \bm z})$ be the solution obtained by Algorithm \ref{alg:bid_aggregation} and $(\bm x^\circ, \bm z^\circ)$ be a solution to the common market problem \eqref{eq:common_market}. It holds that:
{
\begin{align}
	0&\stackrel{(a)}{\leq} 
	\bm c_0^\top (\hat{\bm x}_0-\bm x_0^\circ) + \textstyle\sum_{m\in \mc N^D} \bm c_m^\top (\hat{\bm x}_m -\bm x_0^\circ ) \notag\\
	&\stackrel{(b)}{=}(c_0^\top A_m^{\dagger}B_0 + (\col((c_m^\top A_m^{\dagger}B_m)_{m \in \mc N^{\mathrm D}}))^\top)(\bm z^\circ - \hat{\bm z}) \notag\\
	& \stackrel{(c)}{\leq} \|(c_0^\top A_m^{\dagger}B_0)^\top + \col((c_m^\top A_m^{\dagger}B_m)_{m \in \mc N^{\mathrm D}})\|_\infty \bar{\delta} \notag
	\end{align}}%
	where $(a)$ is due to Proposition \ref{cor:lower_bound_rsf_optimality}; $(b)$ is due to \eqref{eq:power_balance_dso}, for all $m \in \mc N^{\mathrm D}$, and 
	\eqref{eq:power_balance_tso}; and $(c)$ is due to 
	the triangle inequality, 
	Lemma \ref{le:distance_rsf_common_sols}, and the definition of $\bar{\delta}$. \qedd 
	
	\bibliography{IEEEabrv,ref}
	\bibliographystyle{ieeetr}
	\begin{IEEEbiography}[{\includegraphics[width=1in,height
			=1.2in,clip]{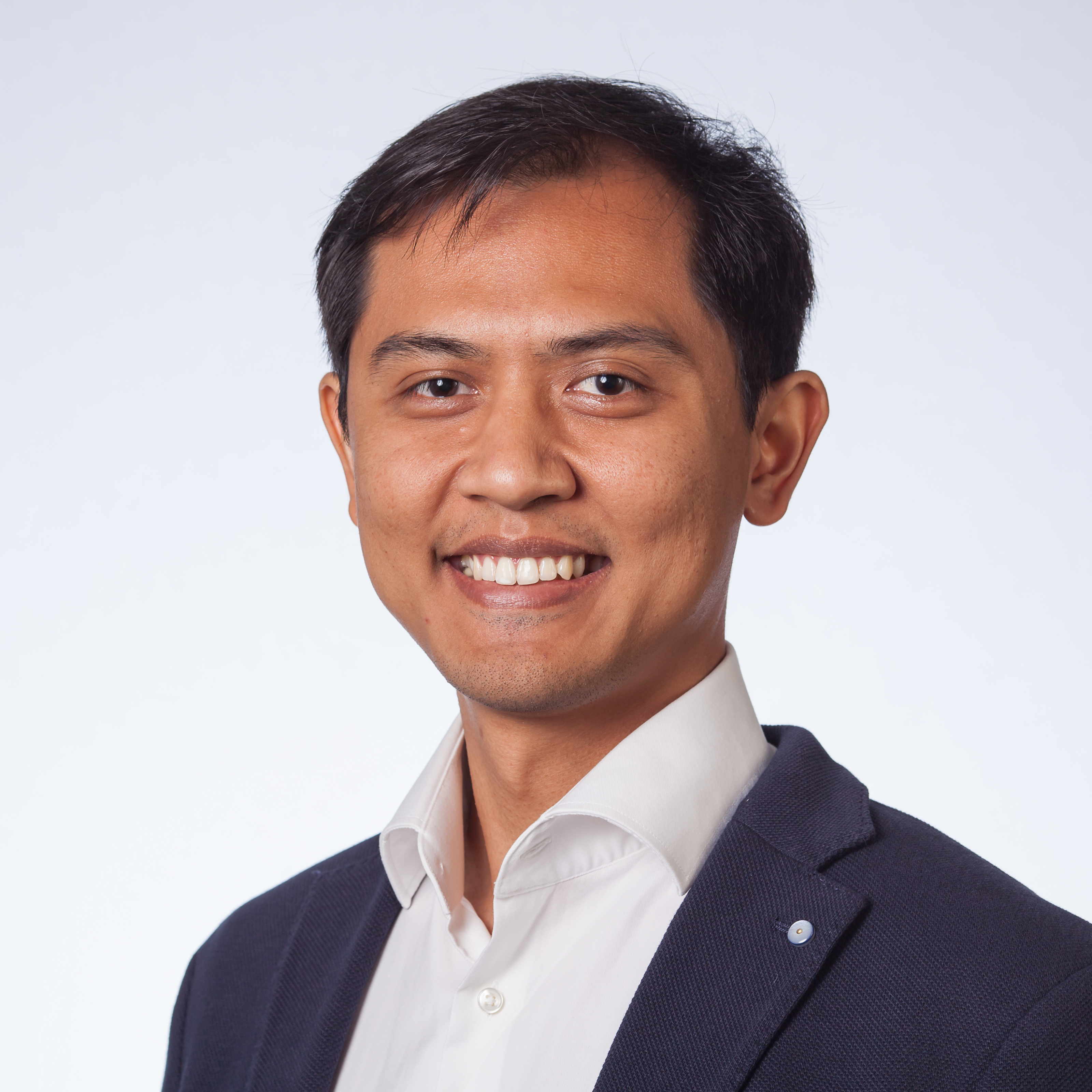}}]{Wicak Ananduta} 
		has been a researcher at the Flemish Institute for Technological Research (VITO), Belgium, since 2023. His work focuses on optimization problems and models for energy markets and systems. He received his Ph.D. degree in Automatic Control, Robotics and Vision from Universitat Polit\`{e}cnica de Catalunya, Spain in 2019, his Master's degree in Systems and Control from TU Delft, The Netherlands in 2016, and his Bachelor's degree in Electrical Engineering from the University of Indonesia in 2011. From 2020-2023, he was a postdoc researcher at the Delft Center for Systems and Control, TU Delft. He was the recipient of the best paper award at the IEEE ISGT Europe conference in 2024.
	\end{IEEEbiography}
	\begin{IEEEbiography}[{\includegraphics[width=1in,height
			=1.2in,clip]{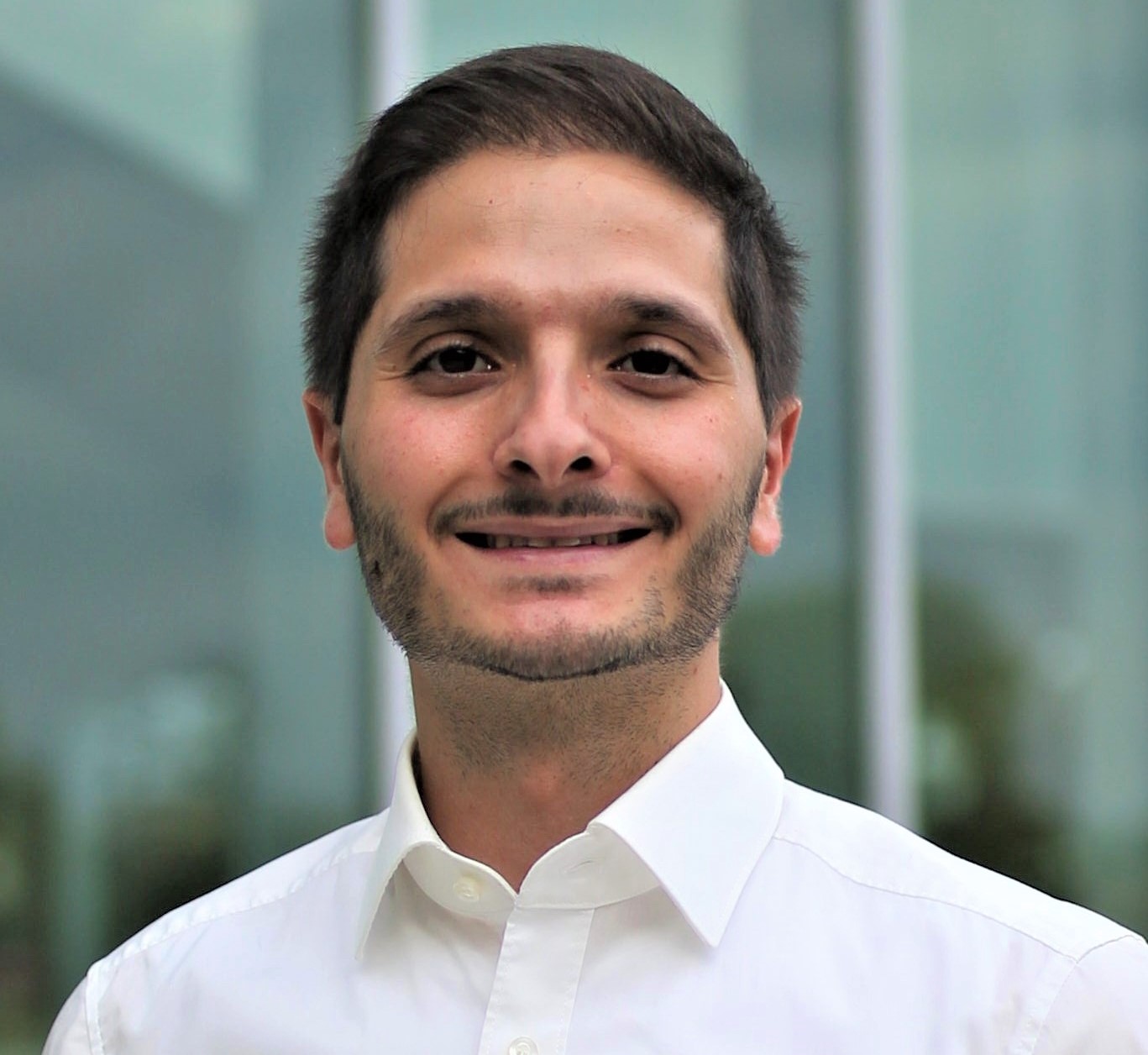}}]{Anibal Sanjab} 
			received the B.E. degree in electrical engineering from the Lebanese American University Byblos, Lebanon, in 2012, and the M.S. and Ph.D. degrees in electrical engineering from Virginia Tech, Blacksburg, VA, USA, in 2014 and 2018, respectively. He has been a Senior Researcher at the Flemish Institute for Technological Research (VITO) and EnergyVille, Belgium, since 2019. His research experience and current interests focus on electricity and energy markets (wholesale and decentralized schemes), renewable energy integration, power systems (operation, planning, and control), flexibility markets and procurement mechanisms, TSO-DSO coordination, sector coupling, and cyber-physical systems operation and security for which he develops mathematical models, tools, and concepts based in game theory, optimization, graph theory, and machine learning. He was awarded the best paper award at the IEEE ISGT Europe in 2024 and the best student paper award at the IEEE SmartGridComm in 2015. He is also the recipient of the Fulbright scholarship for the years 2012-2014. Dr. Sanjab continuously serves on the technical and organizing committees for a number of international conferences and workshops.
	\end{IEEEbiography}
	
	\begin{IEEEbiography}[{\includegraphics[width=1in,height
			=1.2in,clip]{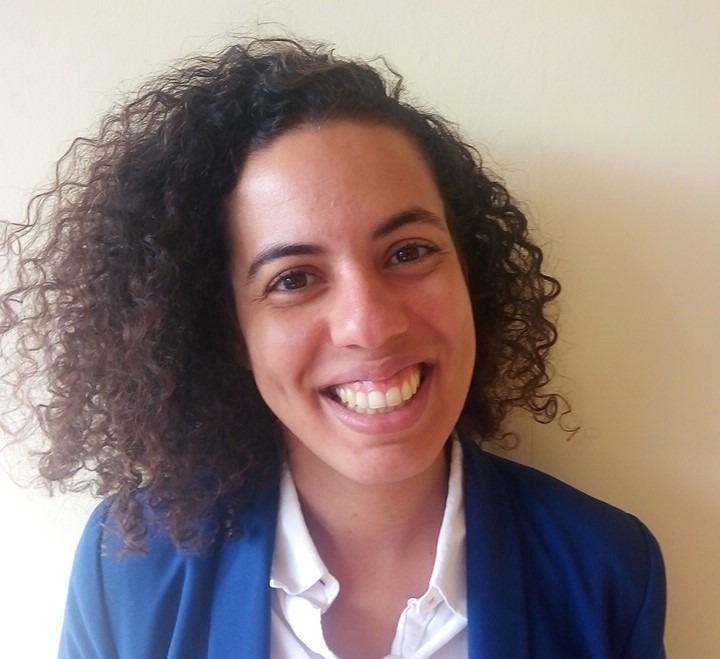}}]{Luciana Marques} 
			holds a Ph.D. in Electrical Engineering (2022) and M.S./B.S. in Industrial Engineering from the Federal University of Minas Gerais, Brazil. She worked at VITO/EnergyVille (BE) from 2021 to January 2025 on European and Belgian energy market projects. Since February 2025, she has been a Senior Energy System Modeler at Open Energy Transition (OET), leading energy modeling studies and open-source tool development. Her expertise includes energy markets, game theory, optimization, and TSO-DSO coordination. She was a Fulbright Scholar at Lawrence Berkeley National Laboratory (2019–2020) and has prior experience as a Business Consultant and Professor.
	\end{IEEEbiography}
	\balance 
\end{document}